\documentclass[12pt]{article}
\usepackage[T1]{fontenc}
\usepackage{mathrsfs}
\usepackage{amssymb}	
\usepackage{amsthm}
\usepackage{xcolor}
\usepackage{amsmath}
\usepackage{multirow}
\usepackage{nccmath}
\usepackage{rotating}
		
\usepackage[left = 2cm, right = 1.5cm, bottom = 5cm, top = 4cm]{geometry}

\DeclareMathOperator   {\tr}      {tr}
\newcommand{\dd}{{\operatorname{d}}}
\newcommand{\ii}{\mathrm{i}}
\newcommand{\R}{\mathbb R}
\newcommand{\Id}{\operatorname{Id}}
\newcommand{\gl}{\operatorname{gl}}
\newcommand{\weg}[1]{}

\theoremstyle{plain}
\newtheorem{Theorem}{Theorem}
\newtheorem{Lemma}{Lemma}

\newtheorem{Remark}{Remark}
\newtheorem{Proposition}{Proposition}

\usepackage{graphicx}
\graphicspath{ {./images/} }

\title{Beltrami problem in dimension two: local normal forms}
\author{ Dinmukhammed Akpan\footnote{Institut f\"ur Mathematik, Friedrich Schiller Universit\"at Jena, 07743 Jena, Germany, Institute of Mathematics and Mathematical Modeling, Almaty, Kazakhstan, dinmukhammed.akpan@uni-jena.de}, Alexey Bolsinov\footnote{School of Mathematics, Loughborough University, Loughborough LE11 3TU, UK and Institute of Mathematics and Mathematical Modeling, Almaty, Kazakhstan, A.Bolsinov@lboro.ac.uk}}

\begin{document}
\maketitle
\tableofcontents

\section{Introduction}\label{sect1}

Two (pseudo-)Riemannian metrics $g, \bar{g}$ are called \textit{geodesically equivalent} if their geodesics coincide as unparametrized curves.
The theory of geodesically equivalent metrics has a deep and extensive history. The earliest known examples were already constructed by Lagrange \cite{lagrange1779}, and foundational contributions were later made by prominent mathematicians such as Beltrami \cite{beltrami1865, beltrami1868, beltrami1869}, Levi-Civita \cite{levicivita1896}, Painlev\'e \cite{painleve1897}, Lie \cite{lie1882}, Liouville \cite{liouville1889}, Fubini \cite{fubini1903}, Eisenhart \cite{eisenhart1927, eisenhart1923}, Weyl \cite{weyl1921}, and Thomas and Veblen \cite{thomas1945, veblen1923, veblen1926}. Between the 1950s and 1990s, the subject became a central focus within the Soviet and Japanese schools of differential geometry (see surveys \cite{aminova2003, mikes1996} and \cite{yano1940, yano1956, yano-nagano1957}). In recent years, the theory has experienced a resurgence, driven by the introduction of modern mathematical techniques, particularly those from the theory of integrable systems \cite{matveevtopalov1998, matveev-topalov2001} and parabolic Cartan geometry \cite{bryant2009, eastwood2007}. These approaches have enabled researchers to solve numerous classical problems over the past decade, including those formulated by Sophus Lie \cite{bryant2008, matveev2012}, the Lichnerowicz conjecture \cite{matveev2007}, and the Weyl--Ehlers problems \cite{kiosak2009, matveev2012_gr}.  More recently, in \cite{BKM-1, Bolsinov-2025, nijapp5} new methods for studying geodesically equivalent metrics were proposed, based on their connection with Nijenhuis geometry.
   

In 1865, E. Beltrami posed the problem of \textit{local description of all geodesically equivalent pairs} \cite{beltrami1865}. The first result was obtained by U. Dini  in 1869 \cite{dini1869}, who described geodesically equivalent Riemannian metrics in a neighborhood of a generic point in dimension two. 

Later, in 1896, T. Levi-Civita generalized Dini's result to arbitrary dimensions \cite{levicivita1896}.  In the pseudo-Riemannian case, the local description of geodesically equivalent metrics at generic points was obtained in \cite{tams}, see also \cite{bolsinov-matveev-pucacco2009} for dimension 2.

For a long time, the question of describing geodesically equivalent metrics in a neighbourhood of a singular point remained open. In 1983, V. Kolokoltsov \cite{kolokoltsov1983}, and later A. Bolsinov, V. Matveev, and A. Fomenko \cite{bolsinov-matveev-fomenko1998}, in the Riemannian case, obtained  a description of singular points  for a geometric problem in dimension 2, which was equivalent to the question asked by Beltrami.    

In this paper, we solve the Beltrami problem in the remaining, most difficult case in dimension 2 and  provide a complete description of geodesically equivalent pseudo-Riemannian metrics in a neighbourhood of a singular point. 

\begin{center}
\footnotesize
In dimension two, we have the following table of results:

\medskip
\renewcommand{\arraystretch}{1.2}
\setlength{\tabcolsep}{6pt}

\begin{tabular}{p{1.4cm}|p{4.5cm}|p{4.8cm}}
    & \textbf{Riemannian} & \textbf{Pseudo-Riemannian} \\
    \hline
    \parbox[t]{1.2cm}{\centering \textbf{Generic \quad }} 
    & U.\,Dini (1869) 
    & A.\,Bolsinov,\,V.\,Matveev, G.\,Pucacco (2009) \\
    \hline
    \parbox[t]{1.2cm}{\centering \textbf{Singular \quad }} 
    & V.\,Kolokoltsov\,(1982), A.\,Bolsinov,\,V.\,Matveev, A.\,Fomenko (1998) 
    & New results \\
\end{tabular}
\end{center}

Instead of dealing with the Beltrami problem directly, we solve an equivalent problem and describe normal forms of pseudo-Riemannian metrics whose geodesic flows admit a first integral quadratic in momenta   (see \cite{kolokoltsov1983, bolsinov-matveev-fomenko1998} for the Riemannian case).  The equivalence between these two problems follows from an important and non-trivial observation due to Matveev and Topalov  \cite{matveev-topalov2001} (see also \cite{bolsinov-matveev-fomenko1998}).  

\medskip

{\bf Fact.}  In dimension 2,  a metric $g$ admits a non-trivial geodesically equivalent metric $\bar g = \sum_{i,j=1}^2 \bar g_{ij}\, \dd x^i \dd x^j$ if and only if the geodesic flow of $g$ admits a non-trivial first integral  $F = \sum_{i,j=1}^2 f^{ij}(x) p_i p_j$ quadratic in momenta.  Moreover, if we think of $F$, $g$ and $\bar g$ as $2\times 2$ symmetric matrices, then 
    \begin{equation}\label{quad_int_and_metric_formula}
         F = \Big( \frac{\det g}{\det \bar{g}} \Big)^{\frac{2}{3}} g^{-1} \bar{g} g^{-1} \quad\mbox{or, equivalently,}\quad  \bar{g} = \frac{1}{(\det g \det F)^2}\, gFg.
    \end{equation}

\medskip

The problem we consider in this paper is to describe locally  the pairs $ (g, F) $, where $g$ is a pseudo-Riemannian metric and $F$ is a quadratic integral of its geodesic flow. More specifically, we address the following question:
\begin{center}
    \textit{What are normal forms for such a metric $g$ and quadratic integral $F$ near a singular point?}
\end{center}

After answering this question, the normal forms for $g$ and $\bar g$ can be immediately obtained by \eqref{quad_int_and_metric_formula}.

\medskip
Let us say a few words about why, in our opinion, this problem has only now been solved.   The singularities we are interested in are primarily related to bifurcations of the spectrum (more precisely, of algebraic type) of the Killing $(1,1)$-tensor  $K$  naturally related to the quadratic integral $F$. In the Riemannian case,  the operator $K$ is self-adjoint and, therefore, diagonalises at every point. In the pseudo-Riemannian case, bifurcations of algebraic type of $K$ are harder to understand due to the presence of Jordan blocks, so the Beltrami problem for singular points seemed to be difficult to attack without additional tools.  Such tools have become available thanks to recent advances in Nijenhuis Geometry \cite{BKM-1, BKM-3, nijapp5}.

\medskip
The key idea is to replace $K$ with an operator $L$ defined by formula \eqref{operator}. This operator has the same algebraic type as $K$, but also   has the additional advantage of being a Nijenhuis operator. This allows us, in particular, to exploit the local classification of $\gl$-regular Nijenhuis operators in dimension 2 \cite{BKM-3}, each of which yields a pair of geodesically equivalent metrics via a certain explicit formula from \cite{nijapp5}.  This direct link between Nijenhuis geometry and the Beltrami problem in dimension 2 provides a fairly clear method for solving the latter. We applied this approach early in our work, which ultimately led us to the structure of the proof presented below.  While the above connection to the results of \cite{BKM-3, nijapp5} remains behind the scenes in the present version of the proof, at each step we compared it with the classification results from \cite{BKM-3} for consistency, and we plan to discuss this important connection in more detail elsewhere.

Throughout this paper, we make two general assumptions:
\begin{itemize}

\item we treat the Beltrami problem in the real analytic setting, that is, all the objects we are dealing with are assumed to be real analytic;

\item we treat the Beltrami problem in the local setting, that is,  the manifold $M^2$ is a small neighbourhood of a point,  e.g., a neighbourhood of $(0,0)\in\R^2$.  

\end{itemize}


The authors sincerely thanks V.\,Matveev, A.\,Konyaev, and A.\,Glutsyuk for their valuable remarks and constructive feedback. This research was supported by the Ministry of Science and Higher Education of the Republic of Kazakhstan (grant No. AP23483476) and by the DFG 529233771.

\section{Admissible singularities}\label{sect2}

Consider a pseudo-Riemannian metric $g=\bigl( g_{ij}
\bigr)$ on $M^2$  and its geodesic flow treated as a Hamiltonian system of $T^*M^2$ with the Hamiltonian $H = \frac{1}{2} \sum_{i,j=1}^2 g^{ij}(x)p_i p_j $, where $ (x_1, x_2, p_1, p_2)$ are canonical coordinates on $ T^*M^2 $.  We assume that this geodesic flow admits a non-trivial first integral 
$F = \sum_{i,j=1}^2 a^{ij}(x)p_i p_j$ quadratic in momenta,  so that $H$ and $F$ commute with respect to the canonical Poisson bracket on $T^*M^2$. By the non-triviality of the integral, we mean that $ F \neq \operatorname{const} \cdot H$. 

Our purpose is to describe local normal forms of $g$ and $F$ at singular points, and we start with explaining what kind of singularities we have in mind.  The Hamiltonian $H$ and the first integral $F$ can be understood as a pair of (contravariant) quadratic forms.  By an appropriate change of a basis, they can be reduced to a certain normal form which basically depends on the algebraic properties of the operator $HF^{-1}$. For instance, if this operator admits two different real eigenvalues, then $H$ and $F$ can be simultaneously diagonalised. We say that a point $\mathsf p\in M^2$ is {\it generic}, if the algebraic type of $HF^{-1}$ does not change in a neighbourhood of $\mathsf p$.  Otherwise, $p$ is considered to be {\it singular}.  The generic case was completely understood in \cite{bolsinov-matveev-pucacco2009}.

In the context of geodesically equivalent metrics, instead of $HF^{-1}$ it is more convenient to consider the operator
\begin{equation} \label{operator}
   L = \det g \cdot \det F \cdot g^{-1} F^{-1}. 
\end{equation}
It differs from $HF^{-1}$ by a scalar factor and therefore has the same algebraic type. However, $L$ possesses one additional remarkable property of being a {\it Nijenhuis operator} \cite{eqm}.

Recall that in dimension 2, operators may have 4 different algebraic types depending on their Jordan normal form over $\R$: 
\begin{itemize}
\item $
\begin{bmatrix}
    \lambda & 0 \\
    0 & \mu
\end{bmatrix}$,  \quad  two different real eigenvalues $\lambda,\mu\in\R$, $\lambda\ne \mu$; 

\item
$
\begin{bmatrix}
    \alpha & \beta \\
    -\beta & \alpha
\end{bmatrix},$ \  two complex conjugate eigenvalues $\lambda=\alpha\pm \ii \beta$, $\beta\ne 0$; 

\item
$\begin{bmatrix}
    \lambda & 1 \\
    0 & \lambda
\end{bmatrix}, 
$
\quad  Jordan block with a real eigenvalue $\lambda$;

\item $\begin{bmatrix}
    \lambda & 0 \\
    0 & \lambda
\end{bmatrix}, 
$
\quad   scalar matrix $\lambda\Id$.
\end{itemize}

In the pseudo-Riemannian case, there are generic points of the first three types \cite{bolsinov-matveev-pucacco2009}.  The scalar type,  in the generic situation, cannot appear as this would immediately imply that $H$ and $F$ are proportional with a constant factor, which contradicts to non-triviality of $F$.   At a singular point, however,  the operator $L$ can be scalar.  Moreover, we shall see below that there are many different scenarios of how the algebraic type may change.  Our first theorem describes all admissible metamorphoses of $L$ at a singular point $\mathsf p$ in terms of the discriminant of its characteristic polynomial. Since our analysis is local, without loss of generality we assume that $M^2 =\R^2(x_1, x_2)$ and $\mathsf p=(0,0)$.

\begin{Theorem} \label{thm1}
Let $g$ be a pseudo-Riemannian metric in dimension 2, $F$ be a quadratic integral of its geodesic flow and  $D = (\tr L)^2 - 4 \det L$ be the discriminant of the charactecteristic polynomial of the  operator $L$ defined by \eqref{operator}. 
    \begin{itemize}
        \item If $D\not\equiv 0$ and  $D =P_k(x_1,x_2) + \ldots$, where $P_k(x_1,x_2)$ is the first non-zero homogeneous polynomial in the Taylor expansion of $D$, $\deg P_k(x_1,x_2) =k$, then 
        \begin{itemize}
            \item[\rm{(i)}] $P_k(x_1,x_2)$ has at most two roots, namely $P_k= \operatorname{const} (a_1 x_1 + a_2 x_2)^m(b_1 x_1 + b_2 x_2)^s$,  $m\ge 0$, $s\ge 0$, $m+s =k$ and $a_i, b_i \in\R$;
            \item[\rm{(ii)}] There exist local coordinates $(x,y)$ such that $P_k(x,y) = \varepsilon x^m y^s$, $\varepsilon = \pm 1$, $m\le s$,
             with  either $m=0$ or $m=1$, or $(m,s)=(2,2)$. 
            \item[\rm{(iii)}] If $m=s=0$,  i.e.,  $D(\mathsf p)\ne 0$, then $\mathsf p$ is generic. 
            \item[\rm{(iv)}] If $m=0$ and $s>0$, then $L(\mathsf p)$ is conjugate to a Jordan block.  
            \item[\rm{(v)}] If $m$ and $s$ are both greater than zero, then  $L(\mathsf p)=\lambda_0 \Id$.  
        \end{itemize}
            \item If $D \equiv 0$,  and $L(\mathsf p) \ne \lambda_0 \Id$, then $\mathsf p$ is generic.

        \item If  $D \equiv 0$,  and $L(\mathsf p) = \lambda_0 \Id$, then $\dd (L^i_j - \frac{1}{2} \tr L \cdot\delta^i_j) |_{\mathsf p} \ne 0$ for some $1\le i,j\le 2$.
 
           \end{itemize}
\end{Theorem}

\begin{proof}
Since the discriminant of the characteristic polynomial of $L$ has an invariant meaning,  we can verify the statements of Theorem \ref{thm1} in our preferred coordinate system.
Throughout the paper, we will be working in the light-like coordinates $(x,y)$, i.e., such that the basis vectors $\partial_x$ and $\partial_y$ are light-like.  In dimension 2, such coordinates exist for any pseudo-Riemannian metric, and in these coordinates,   the metric takes the form $ g = 2\lambda(x, y)\dd x \dd y$, $\lambda(x,y)\ne 0$.

\begin{Lemma}  Let 
$g = 2\lambda(x, y)\dd x \dd y $. Then $F=a(x,y)p_x^2 + b(x,y)p_x p_y  + c(x,y) p_y^2$ is  a quadratic integral of the geodesic flow of $g$ if and only if the functions $\lambda$, $a$, $b$ and $c$ satisfy the following PDE system:
 \begin{equation}
 \label{poisson br}
\begin{cases}
    a_y = 0 \\
    c_x = 0 \\
    \lambda a_x + 2 \lambda_x a+ (\lambda b)_y   = 0 \\
    \lambda c_y + 2\lambda_y c \, + (\lambda b)_x   = 0
\end{cases}
\end{equation}
Moreover, if $a=a(x)$ and $c=c(y)$  (as it follows from the first  two relations), then $\lambda$ satisfies the second order PDE
\begin{equation}
\label{main eq}
 2\bigl(c(y)\lambda_{yy} - a(x)\lambda_{xx}\bigr) + 3\bigl(c'(y)\lambda_y - a'(x) \lambda_x\bigr) + \lambda \bigl(c''(y) - a''(x)\bigr) = 0.
\end{equation}
\end{Lemma}

\begin{proof}
The condition that $F$ is a first integral of the geodesic flow of $g$ is equivalent to the relation $\{H, F\}=0$, where $\{\, , \,\}$ is the canonical Poisson bracket on $T^*M^2$ and $H = \frac{1}{2\lambda} \, p_x p_y$ is the  Hamiltonian of the geodesic flow.  The bracket of two quadratic functions is a homogeneous cubic polynomial in $p_x$ and $p_y$, and the above four relations in \eqref{poisson br} amount to the vanishing of the four coefficients of this polynomial.

Next,  we differentiate the 3rd  and 4th equations of \eqref{poisson br} by $x$ and $y$ respectively and subtract one from the other. As a result, we eliminate $b$ from the equations and obtain  \eqref{main eq}. 
\end{proof}

In the light-like coordinate system, the operator $L$ defined by \eqref{operator} and the discriminant of its characteristic polynomial take the form
\begin{equation}\label{discriminant}
L = \begin{bmatrix}
    \frac{\lambda b}{2} & -a\lambda \\
    -c\lambda & \frac{\lambda b}{2}
\end{bmatrix}, \quad D = (\lambda b)^2 - 4\lambda^2\left(\frac{b^2}{4} - ac\right) = 4\lambda^2(x,y) a(x)c(y).
\end{equation}

If $D\not\equiv 0$,  then the singular points are exactly those where the eigenvalues collide, i.e., $D=0$. Hence, the set of singular points is defined by the equation $a(x) c(y) = 0$.  Let us set $a(x) = x^m \widehat{a}(x)$ and $c(y) = y^s\widehat{c}(y)$  where $\widehat a(0)\ne 0$ and $\widehat c(0)\ne 0$ so that $m$ and $s$ are the orders of zero of the functions $a$ and $c$ at $x=0$ and $y=0$ respectively.  Then the first non-zero homogeneous term in the Taylor expansion of $D$ will be  $P_{m+s}(x,y) = x^m y^s \widehat{a}(0) \widehat{c}(0)$, which proves item (i).  Items (iii), (iv) and (v) immediately follow from the above explicit form of $L$.

Let us prove the statement (ii), which imposes restrictions on the values of $m$ and $s$.    First of all, without loss of generality we may assume that $m\le s$, otherwise we interchange $x$ and $y$.  Next, we use relation \eqref{main eq} as a PDE equation for $\lambda$ and the fact that $\lambda(0,0)\ne 0$.
\weg{ Let us show that for $(m,s) \in \{m \geq 2, s \geq 2 \} \setminus (2,2)$,  all the solutions of (\ref{main eq})  vanish at the origin, i.e.,  $\lambda(0,0) = 0$.}

 After substitution $c(y) = y^s \widehat{c}(y)$, $a(x) = x^m \widehat{a}(x)$, $\widehat{a}(0) \widehat{c}(0) \neq 0$,  the equation \eqref{main eq} becomes
    \begin{equation}
    \label{eq:lambda3}
    2\bigl(y^s\widehat{c} \lambda_{yy} - x^m\widehat{a}\lambda_{xx}\bigr) + 3\bigl([y^{s}\widehat{c}]_y\lambda_y - [x^{m}\widehat{a}]_x\lambda_x\bigr) + \lambda\bigl([y^s\widehat{c}]_{yy} - [x^m\widehat{a}]_{xx}\bigr) = 0.
    \end{equation}
    
    First assume that $m > 2$, $s\ge m > 2$.   Consider the Taylor expansion of \eqref{eq:lambda3} at $(x,y)=(0,0)$ and take the terms corresponding to $y^{s-2}$ and $x^{m-2}$. They are as follows:
    $$
   s(s-1)\lambda(0,0) \widehat c(0) y^{s-2} \quad \mbox{and} \quad    m(m-1)\lambda(0,0) \widehat a(0) x^{m-2}.
    $$   
 Since these terms must vanish identically and  $\widehat c(0)\ne 0$, $\widehat a(0)\ne 0$, we conclude that $\lambda(0,0)=0$ which contradicts the nondegeneracy of $g$.  Thus, $m\le 2$.
 
 If $m=2$ and $s>2$, then the constant term of the Taylor expansion of  \eqref{eq:lambda3} is  $2\lambda(0,0) \widehat a(0)$, which again leads to $\lambda(0,0)=0$.  Hence, this case is impossible either.
    
The case $s=m=2$ is allowed but under one additional condition, as explained in the next lemma.

\begin{Lemma} \label{res lemma}  Let $s=m=2$, i.e., $F = x^2\widehat{a}(x)p^2_x + b(x,y)p_xp_y + y^2\widehat{c}(y)p^2_y$, with $\widehat{a}(0)\ne 0$ and $\widehat{c}(0) \neq 0$. Then  $\widehat{a}(0) = \widehat{c}(0)$.
\end{Lemma}

\begin{proof}
Substituting $a(x) = x^2\,\widehat{a}(x)$, $c(y) = y^2\,\widehat{c}(y)$ into \eqref{main eq} gives
$$
2\lambda(\widehat{a} - \widehat{c}) = 2(y^2\widehat{c} \lambda_{yy} - x^2\widehat{a}\lambda_{xx}) + 3([y^2\widehat{c}]'\lambda_y - [x^2\widehat{a}]'\lambda_x) + \lambda([4y\widehat{c}' + y^2\widehat{c}''] - [4x\widehat{a}' + x^2\widehat{a}'']).
$$

At the point $(x,y)=(0,0)$, we obtain  $2\lambda(0,0)\bigl(\widehat{a}(0) - \widehat{c}(0)\bigr)=0$. Since $\lambda(0,0) \ne 0$,  we get $\widehat{a}(0) = \widehat{c}(0)$, as stated.
\end{proof}

\medskip

It remains to consider the case when $ D \equiv 0 $.     This means that either $a(x)\equiv 0$ or $c(y)\equiv 0$  (but not both, as in this situation $F$ would be proportional to $H$ which contradicts the non-triviality of the integral $F$).   We will assume that $a(x)\equiv 0$  (otherwise we simply interchange $x$ and $y$) and $ c(y) = y^s\,\widehat{c}(y)$, $\widehat c(0)\ne 0$. Let us show that either $ s = 0 $ or $ s = 1 $.

We again consider the equation \eqref{main eq}, which now becomes:
    $$
    s(s-1)\,\widehat{c}(y)\lambda + 3s\, y\, \widehat{c}(y)\lambda_y + 2\,y^2\widehat{c}(y)\lambda_{yy} = 0.
    $$
Substituting $(x, y) = 0$ gives $ s(s-1)\lambda(0, 0) = 0 $, which implies $ s = 0 $ or $ s = 1 $, since $\lambda(0,0)\ne 0$.

Let $s = 0$, that is,  $c(0) \neq 0$.  Then according to \eqref{discriminant}, we have
$$
L = \begin{bmatrix}
    \frac{\lambda b}{2} & 0 \\
    -c \lambda & \frac{\lambda b}{2}
\end{bmatrix}.
$$
Clearly, $L$ is a Jordan block at each point including $\mathsf p=(0,0)$, i.e., $\mathsf p$ is generic, as stated.

Now let $s = 1$,  that is, $c(y) = y\, \widehat{c}(y)$, $\widehat c(0)\ne 0$.  Then $L$ takes the form
$$
L = \begin{bmatrix}
    \frac{\lambda b}{2} & 0 \\
    -y\widehat{c} \lambda & \frac{\lambda b}{2}
\end{bmatrix}
$$
and, at the point $\mathsf p=(0,0)$,  it becomes a scalar operator $L(0,0)=\frac{\lambda b}{2} \Id$. In this case, $\mathsf p=(0,0)$ is singular as $L$ is a Jordan block for $y\ne 0$.   The operator $L - \frac{1}{2} \tr L \cdot \Id= \begin{bmatrix}
  0 & 0 \\
    -y\widehat{c} \lambda & 0
\end{bmatrix}$  vanishes at $\mathsf p$ but, informally speaking, it has zero of order one in the sense that the differential of one of its components is not zero. Indeed, $\frac{\partial}{\partial y} (-y\widehat c \lambda)|_{\mathsf p=(0,0)} = -\widehat c(0) \lambda (0,0) \ne 0$.  This completes the proof.
\end{proof}

\begin{Remark} {\rm Theorem \ref{thm1} describes, in particular, the local structure of the singular set $\mathsf{Sing}\subset \R^2$ where the algebraic type of $L$ changes.  If $D\not\equiv 0$, i.e., generically $L$ has two different eigenvalues, then  locally  $\mathsf{Sing}$ is either a smooth curve,  or a pair of transversally intersecting curves.  The first case corresponds to $(m,s)=(0,s)$ and in the light-like coordinates $(x,y)$ introduced in  the proof of Theorem \ref{thm1},  we have $\mathsf{Sing}=\{ y=0 \}$.   On this line,  the eigenvalues of $L$ collide, and $L$ becomes a Jordan block. If both $m>0$ and $s>0$,  then $\mathsf{Sing} = \{ x=0\}  \cup \{ y = 0\}$.   On each of these lines, $L$ is a Jordan block, but at the intersection point $\mathsf p=(0,0)$,  it takes the form $L = c\cdot\Id$, $c\in\R$. 

If $D \equiv 0$, then generically $L$ is a Jordan block, $\mathsf{Sing}$ is a  smooth curve on which $L$ becomes scalar, that is,  $L|_{\mathsf{Sing}}= f \cdot \Id$, where $f$ is not necessarily constant  (we will show in Theorem \ref{thm3}, case C4, that the directional derivative of $f$ along this curve is different from zero at each point).

}\end{Remark}

\section{Reducing $F$ to a pre-normal form}\label{sect3}

As explained above, the quadratic integral $F$ in a light-like coordinate system takes the form
$$
F(x,y) = a(x) p_x^2 + b(x,y)p_x p_y + c(y)p_y^2,
$$
where one or both functions $a(x)$ and $c(y)$ vanish at a singular point $\mathsf p=(0,0)$. In this section, our goal is to simplify these functions by using suitable coordinate transformations of the form $x_{\mathrm{new}} = u(x)$ and   
$y_{\mathrm{new}} = v(y)$. The new coordinates are still light-like so that all the above computations and arguments remain valid for them. We will also allow ourselves to interchange $x$ and $y$ and make rescaling of $F$ to make the final normal form of $F$ as simple as possible and unique. 

Since the transformations for $x$ and for $y$ are independent of each other,  the problem reduces to finding a normal form for the expression $a(x)p_x^2$,  so that the problem is, in fact, one-dimensional.

\begin{Lemma}\label{vec norm form}
Let  $ a(x) = x^m \widehat{a}(x)$, $\widehat{a}(0) > 0$, and $m \ge 0$. Then there exists a regular coordinate transformation $x_{\mathrm{new}} = u(x)$  after which the expression $a(x)p_x^2$  takes one of the following forms (we use the old notation $x$ for the new coordinate): 
\begin{itemize}
\item[{\rm (a1)}] $p_x^2$, if $m=0$;
\item[{\rm (a2)}]  $\widehat a(0) x^2 p_x^2$, if $m=2$;
\item[{\rm (a3)}]  $ \dfrac{x^{2k}}{(1 + \omega x^{k-1})^2}\, p_x^2 $, where $\omega=\operatorname{Res}_{z = 0}\frac{1}{z^k\sqrt{\widehat{a}(z)}}$, if $m=2k$,  $k=2,3,4, \dots$;
\item[{\rm (a4)}]  $x^m p_x^2$, if $m=2k+1$, $k=0, 1,2, \dots$

\end{itemize}
\end{Lemma}

\begin{proof}
Let $m=2k$, then we can consider the square root $\sqrt{a(x) p_x^2}=x^k\sqrt{\widehat{a}(x)}\, p_x$ and think of this expression as a vector field on $\R$.  The cases $k=0$ and $k=1$ are standard and immediately lead to items (a1) and (a2). For $k>2$, we use the {\it local normal form} theorem (see \cite{ilyashenko2013}, Theorem 4.29)  for a vector field in dimension 1,  which states that our vector field can be reduced to the form $ \frac{x^k}{1 + \omega x^{k-1}} \partial_x $, where $ \operatorname{Res}_{z = 0}\frac{1}{z^k\sqrt{\widehat{a}(z)}} = \omega$. Thus, the even case $ m = 2k $ is resolved.

Consider the remaining case (a4) when $ m = 2k+1 $ is odd. We need to find $x_{\mathrm{new}} = u(x)$, $\frac{du}{\dd x}(0)\ne 0$, satisfying the equation $ \left(\frac{du}{\dd x}\right)^2 = \frac{u^m}{a(x)} $. It is straightforward to verify that the following function is a solution to our equation:
$$
u(x) = x \left( \sum_{j = 0}^{\infty} \tfrac{2 - m}{2j + 2 - m} A_j x^j \right)^{\frac{2}{2 - m}},
$$
where $ A(x) = \frac{1}{\sqrt{\widehat{a}(x)}} = \sum_{j = 0}^{\infty} A_j x^j $ is an analytic function, which easily implies the convergence of  $\sum_{j = 0}^{\infty} \tfrac{2 - m}{2j + 2 - m} A_j x^j$  so that $u(x)$ is real analytic. Notice the importance of the fact that $m=2k+1$ is odd, as otherwise we have division by zero in one of the coefficients of this power series. We also have $A_0\ne 0$ and hence  $\frac{du}{\dd x}(0)\ne 0$ as required. 
\weg{Let $m = 0$, then we have $a(x)p_x^2$ where $a(0) > 0$. In this case, we need to show that there exists a solution $u = u(x)$ to the ODE obtained from the coordinate change rule $a(x)p_x^2 = p_u^2 = p_x^2 \left(\frac{dx}{du}\right)^2$ such that $u'(0) \neq 0$. It is clear that $u = \int \frac{dx}{\sqrt{a(x)}}$ is a smooth solution and $u'(0) = \frac{1}{\sqrt{{a}(0)}} \neq 0$, and in the $u$-coordinate the form $a(x)p_x^2$ becomes $p_u^2$.

Let $m = 2$, then we have $x^2\widehat{a}(x)p^2_x$, where $\widehat{a}(0) > 0$. We need to show that there exists a solution of the ODE $x^2 \widehat{a}(x) = \widehat{a}(0) u^2 (\frac{\dd x}{\dd u})^2$ such that $u = xf(x)$, $f(0) \neq 0$. Indeed, let us consider 
$$
\frac{\dd u}{u \sqrt{\widehat{a}(0)}} = \frac{\dd x}{x \sqrt{\widehat{a}(x)}} = \Big( \frac{1}{x \sqrt{\widehat{a}(0)}} + \beta_1  + \beta_2 x + \ldots \Big) \dd x,
$$
where $\frac{1}{\sqrt{\widehat{a}(x)}} = \frac{1}{\sqrt{\widehat{a}(0)}} + \beta_1 x + \beta_2 x^2 + \ldots$ Taylor expansion of $\frac{1}{\sqrt{\widehat{a}(x)}}$. 

Finally after integrating 
$$
\ln(u) = \ln(x) + \sqrt{\widehat{a}(0)} \beta_1 x + \frac{\sqrt{\widehat{a}(0)} \beta_2 x^2}{2} + \ldots
$$
or
$$
u = x \underbrace{e^{\sqrt{\widehat{a}(0)} \beta_1 x + \frac{\sqrt{\widehat{a}(0)} \beta_2 x^2}{2} + \ldots}}_{f(x)}.
$$
}
\end{proof}

Using Lemma \ref{vec norm form}, we get the following list of pre-normal forms for $F$, where the function $b(x,y)$ is yet unknown; it will be described in the next sections.

\begin{Proposition} \label{int norm forms} In a neighbouhood of the point $\mathsf p=(0,0)$, there exists a light-like coordinate system in which the quadratic integral $F$ (perhaps, after suitable rescaling $F\mapsto \mathrm{const}\cdot F$) takes one of the following forms:  
    \begin{itemize} 
    \item $\mathsf p$ is generic:

    \begin{tabular}{lll}
{\rm (A1)}   &$F = \ p^2_x + b(x,y)p_xp_y + p^2_y$, &\quad $m=s=0$;\\

{\rm (A2)}    &    $F = -p^2_x + b(x,y)p_xp_y + p^2_y$, &\quad $m=s=0$;\\
        
{\rm (A3)}    &  $F =  \ \ b(x,y)p_xp_y + p^2_y$,  & \quad  $a(x)\equiv 0$, $c(0)\ne 0$; \\
     \end{tabular}

    \item $\mathsf{p}$ is singular and $L(\mathsf p)$ is $\operatorname{gl}$-regular (Jordan block)
    
   \begin{tabular}{lll}
{\rm (B1)}      &  $F  = p^2_x + b(x,y)p_xp_y + y^sp^2_y$, & \quad $m=0$, $s \geq 1$;\\

{\rm (B2)}       &  $F = p^2_x + b(x,y)p_xp_y + \frac{y^{2k}}{(1 + y^{k-1})^2} p^2_y$, & \quad $m=0$, $s=2k>2$;\\

{\rm (B3)}         &   $F = -p^2_x + b(x,y)p_xp_y + y^{2k}p^2_y$, & \quad $m=0$, $s=2k$;\\

{\rm (B4)}         &  $F = -p^2_x + b(x,y)p_xp_y + \frac{y^{2k}}{(1 + y^{k-1})^2}p^2_y$, & \quad $m=0$, $s=2k$; \\
    \end{tabular}

    \item  $\mathsf{p}$ is singular and $L(\mathsf p)$ is a scalar operator: 
    
  \begin{tabular}{lll}
{\rm (C1)}        & $F = xp^2_x + b(x,y)p_xp_y + y^sp^2_y$, & \quad $m=1$, $s \geq 1$;\\

{\rm (C2)}         & $F = x p^2_x + b(x,y)p_xp_y + \frac{y^{2k}}{(1 + y^{k-1})^2}p^2_y$, &\quad $m=1$, $s=2k>2$;\\

{\rm (C3)}       & $F = x^2p^2_x + b(x,y)p_xp_y + y^2p^2_y$, & \quad $m=2$, $s=2$; \\

{\rm (C4)}     & $F =  \ \ b(x,y)p_xp_y + yp^2_y$, &\quad $D\equiv 0$, $a(x)\equiv 0$, $c(y) = y\widehat c(y)$.\\
                  \end{tabular}

    \end{itemize}
\end{Proposition}

\begin{proof}

The generic case (A)  is treated in \cite{bolsinov-matveev-pucacco2009}.  It easily follows from
item (a1) of Lemma \ref{vec norm form}.   

Next, as explained in the proof of Theorem \ref{thm1},   we can always choose light-like coordinates $x$ and $y$ in such a way that $F$ takes the form
$$
F = x^m \widehat a(x)p_x^2 + b(x,y)p_xp_y + y^s \widehat c(y) p_y^2,
$$
with $m$ and $s$ specified in item (ii),
or if $D\equiv 0$
$$
F =  b(x,y)p_xp_y + y \widehat c(y) p_y^2.  
$$
with $\widehat a(0) \ne 0$ and $\widehat c(0) \ne 0$.

It remains to use suitable transformations $x_{\mathrm{new}} = u(x)$,   $y_{\mathrm{new}} = v(y)$ to reduce the expressions $x^m \widehat a(x)p_x^2$ and $y^s \widehat c(y) p_y^2$ to their normal forms  from Lemma \ref{vec norm form}, taking into account the list of admissible pairs $(m,s)$ from Theorem \ref{thm1}.

To get the expressions from Proposition \ref{int norm forms},  we will need some additional steps.  Namely, we use the fact that for $s\ne 2$, the expressions of type $y^s p_y^2$ and $c \, y^s p_y^2$, $c>0$, are equivalent in the sense that they can be transformed to each other by a suitable rescaling $y \mapsto \alpha y$.  Moreover, if $s$ is odd, then $y^s p_y^2$ and $-y^s p_y^2$  are equivalent.  We are also allowed to change the sign of $F$ and even rescale it by multiplying with a positive constant.    By using these operations we obtain all the normal forms from Theorem \ref{thm2} except (B2), (B3), (C2) and (C4).

Unlike Lemma \ref{vec norm form},  in  (B2), (B3) and (C2)  the coefficient $\omega$ equals $1$.   We can always achieve this result by rescaling $F$. Indeed, the coefficient $\omega$ is the residue of the function $\frac{1}{\sqrt{z^k\widehat c(z)}}$ so it can be made equal to $1$ by rescaling of $F$.  After this, the coefficient in front of $\pm p_x^2$ or $xp_x^2$ will change,  but we can make it equal $1$ by rescaling $x$.   This argument shows that $\omega$ can always be eliminated.

Finally, in the case  $m=s=2$, we first use item (a2) of Lemma \ref{vec norm form} to reduce $F$ to the form
$$
F = x^2 \widehat a(0) p_x^2 +  b(x,y)p_xp_y + y^2 \widehat c(0) p_y^2.
$$
But according to Lemma \ref{res lemma},  $\widehat a(0) = \widehat c(0)\ne 0$ and we can simultaneously eliminate these coefficients by replacing $F$ with $\frac{1}{\widehat a(0)}F$,  which gives (C3).  \end{proof}

\section{Normal forms}\label{sect4}

We continue working in the light-like coordinate system $(x,y)$ in which
$F = ap_x^2 + bp_xp_y + c p_y^2$,  where $a=a(x)$ and $c=c(y)$ have zeros of order $m$ and $s$ respectively at $(x,y)=(0,0)$.

We start with the case $m=0$, i.e., $a(0)\ne 0$.  In this case, according to Lemma \ref{vec norm form},  we can introduce a new coordinate $x_{\mathrm{new}} = u(x)$  such that in the new coordinates $(x_{\mathrm{new}}, y)$ we have  $a\equiv 1$.  Thus, without loss of generality, we  assume that $F = p_x^2 + bp_xp_y + c p_y^2$.
Taking $c=c(y)\not\equiv 0$ to be an arbitrary function, we want to find a general formula for $\lambda$ and $b$ in terms of the function $c$. 

For $a(x)=1$,  the equation \eqref{main eq}  takes the form
\begin{equation}
\label{eq:mainlambda1}
- \lambda_{xx} + c \lambda_{yy}  + \tfrac{3}{2} c_y \lambda_y  + \tfrac{1}{2}c_{yy}\lambda = 0 .
\end{equation}
Our first goal is to describe all local real analytic solutions to this equation in a neighbourhood of $(x,y)=(0,0)$ under the condition that $c=c(y)$, $c(0)=0$ and $c(y)\ne 0$ for $y\ne 0$.

The line $y=0$ divides the plane $\R^2(x,y)$ into two half-planes.  If $c(y)$ changes sign at $y=0$, then the equation \eqref{eq:mainlambda1} changes its type from hyperbolic to elliptic. For this reason, we treat the half-planes $y>0$ and $y<0$ separately.

First we consider the case when $c(y)$ is positive for $y>0$  (or for $y<0$) and therefore can be written as $c(y) =  f^2(y)$ for a certain real analytic $f(y)$.  We start with a local treatment of the problem.   Consider an arbitrary point $(x,y)$ with $y\ne 0$. In the neighbourhood of such a point, the problem is easy to solve by a suitable coordinate transformation. 

Let us introduce a new auxiliary variable 
\begin{equation}
\label{eq:h}
v= q(y) =  \int \frac{1}{f(y)} \dd y, 
\end{equation}
which is a solution of the equation \begin{equation}
\frac{\dd v}{\dd y}= \frac{1}{f(y)} \quad \mbox{or, equivalently,} \quad  \frac{\dd y}{\dd v} = f(y),  
\end{equation} 

We also introduce a new function
$$
\widetilde\lambda (x,v) = \lambda \bigl(x, y(v)\bigr) f\bigl(y(v)\bigr). 
$$ 

It is easy to check that  (we use the fact that $\frac{\dd y}{\dd v} = f$ and $c=f^2$):
$$
\begin{aligned}
-\widetilde\lambda_{xx} + \widetilde\lambda_{vv} &= -  \lambda_{xx} f 
+ \tfrac{\partial }{\partial v} \left( \lambda_y \tfrac{\dd y}{\dd v}   f  +  \lambda f_y \tfrac{\dd y}{\dd v}  \right)  =  
-  \lambda_{xx} f  +\tfrac{\partial }{\partial v} \left( \lambda_y   f^2  + 
 \tfrac{1}{2}  \lambda  (f^2)_y \right) =   \\
&= -  \lambda_{xx} f  \tfrac{\partial }{\partial v} \left( \lambda_y   c   + 
 \tfrac{1}{2}  \lambda  c_y \right)  =  
  -
  \lambda_{xx} f  +  \left(\lambda_{yy}  c + \lambda_y c_y  + \tfrac{1}{2} \lambda_y c_y + \tfrac{1}{2} \lambda c_{yy} \right) \tfrac{\dd y}{\dd v}   =\\
&=
\left(- \lambda_{xx} + \lambda_{yy}  c  + \tfrac{3}{2} \lambda_y c_y + \tfrac{1}{2} \lambda c_{yy} \right)  f  = 0.
\end{aligned}
$$

The equation 
\begin{equation}
\label{eq:hyper}
-\widetilde\lambda_{xx} + \widetilde\lambda_{vv} = 0
\end{equation}
is easy to solve.  Its general solution has the form
$ \widetilde \lambda(x, v) = \frac{1}{2} \big(\widetilde X(v+x) - \widetilde Y (v-x) \big)$,
where $\widetilde X$ and $\widetilde Y$ are arbitrary analytic functions.  Hence, a general (local) solution of \eqref{eq:mainlambda1}  (for $y\ne 0$) is
$$
 \lambda(x,y) = \frac{1}{2 f(y)} \left( \widetilde X(q(y)+x) - \widetilde Y (q(y)-x)  \right)
 $$

 Let us introduce one additional function $\rho(x,y)$ which is the solution of the equation 
 \begin{equation}
 \label{eq:forrho}
 \frac{\partial \rho}{\partial x} = f(\rho), \quad \rho (0,y) =y.
 \end{equation}
By Cauchy-Kovalevskaya theorem,  $\rho(x,y)$ exists and is real analytic for $y\ne 0$ and $x$ sufficiently small.  Moreover if $c(y)$ does not change sign, the function $f(\rho)$ is analytic in a neighbourhood of $\rho = 0$ and therefore, $\rho(x,y)$ will be analytic in a neighbourhood of $(x,y)=(0,0)$.

\begin{Lemma}\label{lem:1}
Let $X$ and $Y$ be real analytic functions in a neighbourhood of $0$.  Then in the domain $y\ne 0$, the function
\begin{equation}
\label{eq:lambda}
\lambda(x,y) = \frac{1}{2} \, \frac{X \bigl(\rho ( x,y)\bigr) - Y \bigl(\rho(-x, y)\bigr)}{f(y)}
\end{equation}
is a solution of \eqref{eq:mainlambda1}.
\end{Lemma}

\begin{proof}
Due to the relationship between \eqref{eq:mainlambda1} and \eqref{eq:hyper}, it is sufficient to prove that 
$2\,\widetilde\lambda ( x, v ) = X \bigl(\rho ( x,y(v))\bigr) - Y \bigl(\rho(-x, y(v))\bigr)$ satisfies \eqref{eq:hyper}.

It follows from the definition of $\rho$ and \eqref{eq:h} that
$$
x = \int \frac{1}{f(\rho)} \, \dd \rho + C(y) = q(\rho) + C(y).
$$
Using the initial condition $\rho(0, y) = y$, we get
$0 = q(y) + C(y)$, or equivalently, $C(y) = - q(y)$. Hence, we get the relation
$$
q(\rho) = q(y) + x = v+x.
$$
We can rewrite it as 
$$
\rho \bigl(x,y(v)\bigr) = q^{-1} (v + x)
$$
Thus,  the function 
$2\,\widetilde\lambda ( x, v ) = X \bigl(\rho ( x,y(v))\bigr) - Y \bigl(\rho(-x, y(v))\bigr)$
can be written in the form 
$$
2\,\widetilde\lambda ( x, v ) = X \circ q^{-1} (v + x) - Y \circ q^{-1} (v - x) = \widetilde X (v+x) - \widetilde Y(v-x)
$$
and, therefore, satisfies \eqref{eq:hyper}, which completes the proof. 
\end{proof}

\begin{Remark}\label{rem:1}{\rm
The function $\rho(x,y)$ satisfies one more differential relation which will be used below. It follows from $q(\rho) = q(y) + x$ that 
$q'(\rho) \rho_y = q'(y)$.  Recalling that $q' = \frac{1}{f}$ and $f(\rho) = \rho_x$, we get 
$$
\frac{\rho_y}{\rho_x} = \frac{1}{f(y)} \quad \mbox{or, equivalently,} \quad \rho_x = f(y) \, \rho_y.
$$
}\end{Remark}

It remains to choose $X$ and $Y$ in such a way that $\lambda (x,y)$ has no singularity for $y=0$ and agrees with given initial conditions $\lambda(0,y)=h(y)$ and $\lambda_x(0,y)=H'(y)$, where $h$ and $H$ are real analytic in a neighbourhood of zero.
Let us set
\begin{equation}
\label{eq:ini}
\begin{aligned}
X(\rho) &=  \ \ f(\rho) h(\rho) + H(\rho) \\
Y(\rho) &= - f(\rho) h(\rho) + H(\rho) 
\end{aligned}
\end{equation}

Then for $x=0$ we have
$$
\lambda (0,y) = \frac{X \bigl(\rho ( 0,y)\bigr) - Y \bigl(\rho(0, y)\bigr)}{2 f(y)} =
\frac{X (y) - Y (y)}{2 f(y)} = \frac{f(y) h(y) + f(y) h(y)}{2 f(y)} =  h(y).
$$

Let us compute $\lambda_x (0,y)$.  Since the combination of the terms with $h$ is an even function in $x$, they do not contribute to the final result and we get
$$
\lambda_x (0,y) = \frac{2H'(\rho(0,y))\rho_x(0,y)}{2f(y)} =  \frac{2H'(y) f(y)}{2f(y)}=H'(y),
$$
as required.  

We know from Cauchy-Kovalevskaya theorem that there exists a unique solution of \eqref{eq:mainlambda1} with the initial conditions $\lambda(0,y)=h(y)$ and $\lambda_x(0,y)=H'(y)$. The formulas \eqref{eq:lambda} and  \eqref{eq:ini}
gives such a solution in the domain $\{ y\ne 0\}$.  Hence in this domain, this solution coincides with the one provided by the Cauchy-Kovalevskaya theorem, and therefore, by continuity,  the function $\lambda$ given by \eqref{eq:lambda}, \eqref{eq:ini} extends to the singular line $\{y=0\}$ up to a real analytic solution of \eqref{eq:mainlambda1} in a neighbourhood of $(x,y)=(0,0)$.

In the case,  when $c(y) < 0$  for $y\ne 0$, the situation is similar. The solution of \eqref{eq:mainlambda1} can be found by the same formula,  namely
\begin{equation}
\label{eq:lambdareal}
\begin{aligned}
\lambda(x,y) &= \frac{X \bigl(\rho ( x,y)\bigr) - Y \bigl(\rho(-x, y)\bigr)}{2f(y)}\\
X(\rho) &=  \ \ f(\rho) h(\rho) + H(\rho) \\
Y(\rho) &= - f(\rho) h(\rho) + H(\rho) 
\end{aligned}
\end{equation}
where $f(y) = \sqrt{c(y)}$  and $\rho(x,y)$ solves \eqref{eq:forrho}. But now some of the ingredients of this formula become complex valued.  In particular, the function $f(y)=\ii  \widetilde f(y)$ is pure imaginary, where $\widetilde f^2(y) =-c(y) > 0$.  Hence, the equation \eqref{eq:forrho} takes the form  
$$ 
\frac{\partial \rho}{\partial x} = \ii\, \widetilde f(\rho), \quad \rho (0,y) = y.
$$
and its solution is a complex valued function which can be written as $\rho(x,y) = \widetilde \rho(\ii x, y)$, where $\widetilde\rho(x,y)$ solves the real equation 
\begin{equation}
\label{eq:fortilderho}
\frac{\partial \widetilde \rho}{\partial x} = \widetilde f(\widetilde \rho), \quad \widetilde \rho (0,y) = y
\end{equation}
and, therefore, is real-valued.   In terms of real-valued functions $\widetilde \rho(x,y)$ and $\widetilde f(y)$, we can rewrite \eqref{eq:lambdareal} as follows  
\begin{equation}
\label{eq:lambdacomplex}
\begin{aligned}
\lambda(x, y) &= \frac{X \bigl(\widetilde \rho ( {\mathrm i}x, y)\bigr) - Y \bigl(\widetilde \rho(-{\mathrm i} x, y)\bigr)}{2\,{\mathrm i} \widetilde f(y)}\\
X(\rho) &=  \ \ {\mathrm i} \widetilde f(\rho) h(\rho) + H(\rho) \\
Y(\rho) &= - {\mathrm i} \widetilde f(\rho) h(\rho) + H(\rho) 
\end{aligned}
\end{equation}

We notice that $\lambda(x,y)$ defined by  \eqref{eq:lambdareal} or, equivalently, by \eqref{eq:lambdacomplex}  solves \eqref{eq:mainlambda1} with prescribed initial conditions in the case $c(y)<0$ (elliptic case). This basically follows from the fact that in the real analytic case the passage form the hyperbolic equation \eqref{eq:lambdareal} (with $c(y)>0$) to the elliptic equation \eqref{eq:lambdareal} (with $c(y)<0$) can be done by formal replacement $x$ with $\ii x$.

We also notice that  \eqref{eq:lambdacomplex} can be written in a more elegant form if we introduce complex valued function
$W(v) =  \ii \widetilde f(v) h(v) + H(v)$.  Then \eqref{eq:lambdacomplex} takes the form
\begin{equation}
\label{eq:lambdacomplex2}
\lambda(x,y) = \frac{\operatorname{Im} \Bigl( W\bigl(\widetilde \rho\,( \ii x, y)\bigr)\Bigl)}{\widetilde f (y)},
\end{equation}
where $\widetilde \rho$ is defined by \eqref{eq:fortilderho}.

Next, for a given $\lambda(x,y)=\frac{X \bigl(\rho ( x,y)\bigr) - Y \bigl(\rho(-x, y)\bigr)}{2f(y)}$ we want to reconstruct the function $b(x,y)$ from the PDE system \eqref{poisson br}  (with $a(x)=1$ and $c(y)=f^2(y)$), namely  
$$
\begin{aligned}
(-\lambda b)_x &= 2f f_y\lambda  + 2f^2\lambda_y = 2f (\lambda f)_y = f \Bigl(X \bigl(\rho ( x,y)\bigr) - Y \bigl(\rho(-x, y)\bigr)\Bigr) ;\\
(-\lambda b)_y &=  2\lambda_x;
\end{aligned}
$$
We claim that 
\begin{equation}
\label{eq:forb}
b =   - 2 f(y) \frac{X \bigl(\rho ( x,y)\bigr) + Y \bigl(\rho(-x, y)\bigr)}{X \bigl(\rho ( x,y)\bigr) - Y \bigl(\rho(-x, y)\bigr)},
\end{equation}
or shortly $-\lambda b = X+Y$, $X = X \bigl(\rho ( x,y)\bigr)$ and $Y=Y \bigl(\rho(-x, y)\bigr)$, solves this system.   So we need to show that
$$
\left\{\begin{aligned}
(X+Y)_x &=  f (X-Y)_y \\
(X+Y)_y &= 2 \left(\frac{X-Y}{2f}\right)_x = \frac{1}{f}  (X-Y)_x
\end{aligned}\right.
\quad \mbox{or, equivalently,}\quad \left\{\begin{aligned} X_x &= \ f X_y \\ Y_x &= - f Y_y \end{aligned}\right. .
$$

Since $X_x = X'_\rho \frac{\partial \rho}{\partial x}$ and $X_y = X'_\rho \frac{\partial \rho}{\partial y}$, and similarly 
$Y_x = -Y'_\rho \frac{\partial \rho}{\partial x}$ and $Y_y = Y'_\rho \frac{\partial \rho}{\partial y}$,  this latter system 
is equivalent to
$$
\frac{\partial\rho}{\partial x} = f\,  \frac{\partial\rho}{\partial y},
$$
which is true by Remark \ref{rem:1}.

Formula \eqref{eq:forb}  works for both $c(y)>0$ and $c(y)<0$. However, in the latter case it is easy to see that  \eqref{eq:forb} can be written, in the notation used in \eqref{eq:lambdacomplex2}, as follows
\begin{equation}
\label{eq:forbcomplex}
b(x,y)= - 2 \, \widetilde f(y) \, \frac
{\operatorname{Re}\Bigl( W\bigl(\widetilde \rho\,( \ii x, y)\bigr)\Bigl)}
{\operatorname{Im}\Bigl( W\bigl(\widetilde \rho\,( \ii x, y)\bigr)\Bigl)}
\end{equation}

It is clear that for given $\lambda(x,y)$, $a(x)$ and $c(y)$, the function $b$ is defined up to adding an expression of the form $ \frac{C}{\lambda}$ with constant $C\in\R$.  This freedom can be controlled by a suitable choice of functions $X$ and $Y$.  Indeed, if in \eqref{eq:lambdareal} we change $H$ by $H-\frac{C}{4}$, then $\lambda(x,y)$ remains unchanged but $b=-\frac{2}{\lambda} (X+Y)$  becomes $b_{\mathrm{new}}=b+ \frac{C}{\lambda}$.  Thus,  we don't miss any solutions.

We summarise the above arguments in the following conclusion.

\begin{Proposition}
\label{prop:2}  Let $c(y)$ be a real analytic function and 
$F=p_x^2 + b(x,y)p_xp_y + c(y) p_y^2$ be an integral of the geodesic flow of the metric $g=2\lambda(x,y)\dd x\dd y$ {\rm(}in a neighbourhood of $\mathsf p =(0,0)${\rm)}.  Assume that $\lambda(0,y) = h(y)$, $\lambda_x(0,y) = H'(y)$ and $b(0,0)=-4\frac{H(0)}{h(0)}$, where $h$ and $H$ are arbitrary real analytic functions defined in a neighbourhood of zero, $h(0)\ne 0$. Then locally $\lambda(x,y)$ and $b(x,y)$ are defined by \eqref{eq:lambdareal} and \eqref{eq:forb} respectively.
These formulas define $\lambda$ and $b$ for $y\ne 0$ and can be extended to the line $y=0$ by continuity up to real analytic functions in a neighbourhood of $\mathsf p=(0,0)$.    

If $c(y)<0$, these formulas can be equivalently written in the form \eqref{eq:lambdacomplex2} and \eqref{eq:forbcomplex}.

\end{Proposition}

\begin{Remark} \label{rem:3}{\rm
The statement of Proposition \ref{prop:2} makes sense for $g$ which becomes degenerate at the point $(0,0)$, i.e. if $\lambda(0,0)=h(0)=0$.  The condition $h(0)\ne 0$ is important only to make sure that $g$ is non-degenerate, but is not used in the proof. Of course, in this case the function $b$ does not have to be analytic at those points where $\lambda(x,y)=0$, but $b \lambda$ is still analytic by construction.  In particular,  if we assume that $\lambda(0,y) = h(y) \equiv 0$ and $\lambda_x(0,y) = H'(y)$ with $H'(0)\ne 0$, then $\lambda$ and $b$ will be necessarily of the form     
\begin{equation}
        \label{eq:C??} 
        \begin{aligned}    
        \lambda &= \frac{X \bigl(\rho ( x,y)\bigr) - X \bigl(\rho(-x, y)\bigr)}{f(y)} \\ 
        b &= -2 f(y) \frac{X\bigl(\rho ( x,y)\bigr) + X\bigl(\rho(-x, y)\bigr)}{X\bigl(\rho ( x,y)\bigr) - X\bigl(\rho(-x, y)\bigr)}                \end{aligned}
        \end{equation} 
where $X(\rho)= H(\rho)$.
}\end{Remark}

We are now ready to describe normal forms for $g$ and $F$ in the $\gl$-regular case, i.e., when the eigenvalues of $L$ collide at the singular point $\mathsf p$ in such a way that $L(\mathsf p)$ becomes a Jordan block  (case B in Proposition \ref{int norm forms}).

\begin{Theorem} \label{thm2}
Let $g$ be a two-dimensional pseudo-Riemannian metric,  $F$ a non-trivial quadratic integral of its geodesic flow and $L$ be the operator defined by \eqref{operator}.  Suppose that at a singular point $\mathsf p$,  the operator $L(\mathsf p)$ is $\operatorname{gl}$-regular  \footnote{This condition means that generically the eigenvalues of $L$ are distinct, but at the point $\mathsf p$ they collide and $L(\mathsf p)$ becomes similar to a Jordan $2\times 2$ block.}. Then  $g$ and $F$ {\rm(}perhaps,  up to rescaling $F\mapsto \operatorname{const}\cdot F${\rm)}  can be brought to one of the following normal forms:

    \begin{itemize}
    \item[{\rm (B1.1)}] $m=0$, $s=1$\,{\rm :}
    \begin{equation}\label{eq:B1.1}   
\begin{aligned}    
g &= \frac{X\left(\tfrac{x}{2} + \sqrt{y}\right) - X\left(\tfrac{x}{2} - \sqrt{y}\right)}{\sqrt{y}} \, \dd x \dd y, \\
F &= p^2_x + yp^2_y -2\sqrt{y} \, \frac{X\left(\tfrac{x}{2} + \sqrt{y}\right) + X\left(\tfrac{x}{2} - \sqrt{y}\right)}{X\left(\tfrac{x}{2} + \sqrt{y}\right) - X\left(\tfrac{x}{2} - \sqrt{y}\right)}\, p_xp_y,
\end{aligned}
\end{equation}
    where $X(t)$ is an arbitrary analytic function in a neighbourhood of zero such that $X'(0) \neq 0$.

    \item[{\rm (B1.2)}]  $m=0$, $s=2$\,{\rm :}
\begin{equation}\label{eq:B1.2}   
\begin{aligned} 
g &=  \frac{X(ye^x) - Y(ye^{-x})}{y} \, \dd x \dd y, \\
       F &= p_x^2 + y^2 p_y^2  -2\, y \, \frac{X(ye^x) + Y(ye^{-x})}{X(ye^x) - Y(ye^{-x})},
\end{aligned}
\end{equation}
      where $X(v)=vh(v)+H(v)$ and $Y(v)=-vh(v)+ H(v)$.
      
        \item[{\rm (B1.3)}] $m=0$, $s>2$\,{\rm :}       
\begin{equation}\label{eq:B1.3}
\begin{aligned}
g &=  \frac{X\left(y\bigl(1 + x y^{\frac{s - 2}{2}}\bigr)^{\frac{2}{2 - s}}\right) - Y\left(y\bigl(1 - x y^{\frac{s - 2}{2}}\bigr)^{\frac{2}{2 - s}}\right)}{ \frac{2}{2-s} \, y^{\frac{s}{2}}} \, \dd x \dd y \\
F &= p^2_x + \left(\tfrac{2}{2 - s}\right)^2 y^s p^2_y - \tfrac{4}{2 - s}\,y^{\frac{s}{2}} \frac{
X\left(y\bigl(1 + x y^{\frac{s - 2}{2}}\bigr)^{\frac{2}{2 - s}}\right) + 
Y\left(y\bigl(1 - x y^{\frac{s - 2}{2}}\bigr)^{\frac{2}{2 - s}}\right)}{
X\left(y\bigl(1 + x y^{\frac{s - 2}{2}}\bigr)^{\frac{2}{2 - s}}\right) - 
Y\left(y\bigl(1 - x y^{\frac{s - 2}{2}}\bigr)^{\frac{2}{2 - s}}\right)}\, p_xp_y,
\end{aligned}
\end{equation}
where $X(v) = \frac{2}{2-s} v^{\frac{s}{2}}h(v) + H(v)$, $Y(v) = - \frac{2}{2-s}v^{\frac{s}{2}} h(v) + H(v)$.

        \item[{\rm (B2)}] $m=0$, $s=2k$, $k \geq 2$\,{\rm :}
\begin{equation}\label{eq:B2}
\begin{aligned}
g &= \frac{1 + y^{k-1}}{y^k} \,\Big(X\bigl(\rho(x,y)\bigr) -  Y\bigl(\rho(-x,y)\bigr)\Big) \,\dd x \dd y,\\
F &= p^2_x + \frac{y^{2k}}{(1 + y^{k-1})^2} \, p^2_x - 2 \, \frac{y^k}{1 + y^{k-1}} \,\frac{X\bigl(\rho(x,y)\bigr) +  Y\bigl(\rho(-x,y)\bigr)}{X\bigl(\rho(x,y)\bigr) -  Y\bigl(\rho(-x,y)\bigr)}\, p_xp_y,
\end{aligned}
\end{equation}
where $X(v) = \frac{v^k}{1 + v^{k-1}}h(v) + H(v)$, $Y(v) = -\frac{v^k}{1 + v^{k-1}}h(v) + H(v)$, and   $\rho(x,y)$ is the solution of \eqref{eq:forrho} with $f(y)=\frac{y^k}{1 + y^{k-1}}$.
    \end{itemize}

    \begin{itemize}
         \item[{\rm (B3.1)}]  $m=0$, $s=2$\,{\rm :}
\begin{equation}\label{eq:B3.1}
\begin{aligned}
g &=  2\frac{ \operatorname{Im}\Big( W(ye^{\ii x})\Big)}{y} \, \dd x \dd y, \\
F &= -p^2_x + y^2\, p^2_y + 2\, y\,  \frac{\operatorname{Re}\Big( W(ye^{\ii x})\Big)}{\operatorname{Im}\Big( W(ye^{\ii x})\Big)} \, p_xp_y.
\end{aligned}
\end{equation}
where $W(v) = \ii v h(v) + H(v)$.

         \item[{\rm (B3.2)}]  $m=0$, $s=2k$,  $k \geq 2$\,{\rm :}
\begin{equation}\label{eq:B3.2}
\begin{aligned}
g &=  2\frac{ \operatorname{Im} \Big( W\bigl(y(1 + \ii x y^{k-1})^{\frac{1}{1-k}}\bigr)\Big)}{\frac{1}{1-k} \, y^k} \dd x \dd y,\\
F& = -p^2_x + \left(\frac{y^{k}}{1-k}\right)^2  p^2_y + 2 \, \frac{y^k}{1-k} \, \frac{\operatorname{Re}
\Big( W\bigl(y(1 + \ii x y^{k-1})^{\frac{1}{1-k}}\bigr)\Big)}{\operatorname{Im} 
\Big( W\bigl(y(1 + \ii x y^{k-1})^{\frac{1}{1-k}}\bigr)\Big)} p_xp_y,
\end{aligned}
\end{equation}
where $W(v) =  \ii \,  \frac{v^k}{1-k} \, h(v) + H(v)$.

     \item[{\rm (B4)}]  $m=0$, $s=2k$, $k\ge 2$\,{\rm :}

\begin{equation}\label{eq:B4}
\begin{aligned}
g &=   2 \, \frac{1 + y^{k-1}}{y^k} \operatorname{Im}\Big( W\bigl(\rho(\ii x,y)\bigr)\Big) \dd x \dd y,\\
F &= -p^2_x + \frac{y^{2k}}{(1 + y^{k-1})^2}\, p^2_y + 2 \, \frac{y^k}{1 + y^{k-1}} \, \frac{\operatorname{Re}\Big( W\bigl(\rho(\ii x,y)\bigr)\Big)}{\operatorname{Im}\Big( W\bigl(\rho(\ii x,y)\bigr)\Big)} \, p_xp_y.
\end{aligned}
\end{equation}
where $W(v) =  \ii \, \frac{v^k}{1 + v^{k-1}} \, h(v) + H(v) $,  and  $\rho(x,y)$ is the solution of \eqref{eq:forrho} with $f(y)=\frac{y^k}{1 + y^{k-1}}$.

\end{itemize}

In all the cases,  $h$ and $H$ are arbitrary real analytic functions in a neighbourhood of zero, and $h(0) \neq 0$,  which can be understood as initial conditions for $g=2\lambda \dd x\dd y$, namely, $\lambda(0,y)=h(y)$ and $\lambda_x(0,y)=H'(y)$.
\end{Theorem}

\begin{proof}
The proof follows directly from Propositions \ref{int norm forms} and \ref{prop:2}.  We only need to give some additional comments in each case.

In terms of Proposition \ref{int norm forms},  we consider the case B.  Then according to this proposition, the coordinates $x$ and $y$ can be chosen in such a way $F=\pm p_x^2 + b(x,y)p_xp_y + c(y)p_y^2$ with some explicitly given function $c(y)$. 
In this setting,  $\lambda$ and $b$ can be reconstructed by Proposition \ref{prop:2}. We only need to describe the corresponding functions $f(y)$ and $\rho(x,y)$.

Some of the cases from Proposition \ref{int norm forms} will be divided into subcases.  Here is the list of all essentially different (sub)cases. 

\begin{itemize}
\item [{\rm (B.1.1)}] $m=0$, $s=1$\,{\rm :}  
$\varepsilon = 1$,  $c(y) = y$. Then $f(y)=\sqrt{y}$, $\rho(x,y)=\left(\tfrac{x}{2} + \sqrt{y}\right)^2$.    If we denote $t=\tfrac{x}{2} + \sqrt{y}$, then  $X(\rho) = f(\rho)h(\rho) + H(\rho) = t \, h(t^2) + H(t^2)=\widetilde X(t)$ can be understood as an arbitrary analytic function of $t$.  Then $Y(\rho) = - f(\rho)h(\rho) + H(\rho) = - t \, h(t^2) + H(t^2)=\widetilde X(-t)$.  Hence, $X\bigl(\rho(x,y)\bigr)=\widetilde X \left(\tfrac{x}{2} + \sqrt{y}\right)$ and $Y\bigl(\rho(-x,y)\bigr)=\widetilde X \left(\tfrac{x}{2} - \sqrt{y}\right)$ and  \eqref{eq:B1.1} follows from general formulas \eqref{eq:lambdareal} and \eqref{eq:forb} in Proposition \ref{prop:2}  (where we replace $\widetilde X(\cdot)$ by $X(\cdot)$ to simplify the final notation).  

\item [{\rm (B.1.2)}] $m=0$, $s=2$\,{\rm :} $\varepsilon = 1$, $c(y) = y^2$. Then  $f(y) = y$, $\rho(x,y) = ye^x$, and \eqref{eq:B1.2} follows from general formulas \eqref{eq:lambdareal} and \eqref{eq:forb} in Proposition \ref{prop:2}.

\item [{\rm (B.1.3)}] $m=0$, $s>2$\,{\rm :} 
$\varepsilon = 1$, $c(y)= \left( \frac{2}{2-s}\right)^2  y^s$  (here we rescale $c(y)$ in formula (B2) from Proposition \ref{int norm forms} to simplify coefficients in the final formula \eqref{eq:B1.3}).  Then  $f(y) = \left( \frac{2}{2-s}\right)^2  y^{\frac{s}{2}}$, $\rho(x,y) = y \left(  1 + x y^{\frac{s-2}{2}}  \right)^{\frac{2}{s-2}}$ and \eqref{eq:B1.3} follows from general formulas \eqref{eq:lambdareal} and \eqref{eq:forb} in Proposition \ref{prop:2}.

\item [{\rm (B.2)}] $m=0$, $s=2k$, $k\ge 2$\,{\rm :} $\varepsilon = 1$,
$c(y)= \frac{y^{2k}}{(1 + y^{k-1})^2}$. Then $f(y) =  \frac{y^{k}}{1 + y^{k-1}}$. The solution of  \eqref{eq:forrho} cannot be expressed in terms of elementary functions, and \eqref{eq:B2} is left in the general form from Proposition \ref{prop:2}, where $\rho$ is the solution of \eqref{eq:forrho} with $f(y) =  \frac{y^{k}}{1 + y^{k-1}}$. 
\end{itemize}

The remaining cases (B.3.1), (B.3.2) and (B.4) are {\it complex} analogs of (B.1.2),  (B.1.3) with $s=2k$, and (B.2) respectively.  To formally apply Proposition  \ref{prop:2}, we can change the sign of $F$, which will lead to the change of the sign of $c(y)$.   As explained above, this will result in changing $\rho(x,y)$ with $\rho(\ii x, y)$ and $f(\cdot)$ with $\ii f(\cdot)$,  where $\rho$ and $f$ are the {\it old}  functions from (B.1.2),  (B.1.3) and (B.2).  After this operation, we change the sign of $F$ again to get the final normal forms 
\eqref{eq:B3.1},  \eqref{eq:B3.2} and  \eqref{eq:B4}. \end{proof}

Finally we describe the normal forms for case $C$ from Proposition \ref{int norm forms}.

\begin{Theorem}\label{thm3}
Let $g$ be a two-dimensional pseudo-Riemannian metric,  $F$ a non-trivial quadratic integral of its geodesic flow and $L$ be the operator defined by \eqref{operator}. Suppose that at a singular point $\mathsf p$, the operator  $L(\mathsf p)$ is  scalar, i.e., $L(\mathsf p) = \lambda_0\,\Id$. Then $g$ and $F$ (perhaps up to rescaling $F \mapsto \mathrm{const}\cdot F$) can be brought to one of the following normal forms:
    \begin{itemize}
        \item[{\rm (C1.1)}]   $m=1$, $s=1$:   
        \begin{equation}
        \label{eq:C1.1} 
        \begin{aligned}                             
        g &= \frac{X\left(\sqrt{x} + \sqrt{y}\right) - X\left(\sqrt{x} - \sqrt{y}\right)}{\sqrt{x}\sqrt{y}}\, \dd x \dd y, \\
         F &= x\, p^2_x  + y\, p^2_y - 2\sqrt{x}\sqrt{y} \, \frac{X\left(\sqrt{x} + \sqrt{y}\right) + X\left(\sqrt{x} - \sqrt{y}\right)}{X\left(\sqrt{x} + \sqrt{y}\right) - X\left(\sqrt{x} - \sqrt{y}\right)}\,p_xp_y, 
        \end{aligned}
        \end{equation} 
  where      $X(t)$ is an even analytic function near zero  such that $X''(0) \neq 0$;
        
        \item[{\rm (C1.2)}]   $m=1$, $s=2$:
        
        \begin{equation}
        \label{eq:C1.2} 
        \begin{aligned}    
        g &= \frac{X\left(y e^{2\sqrt{x}}\right) - X\left(ye^{-2\sqrt{x}}\right)}{y\sqrt{x}}\, \dd x \dd y, \\ 
        F &= x\, p^2_x + y^2\, p^2_y -2y\sqrt{x} \, \frac{X\left(y e^{2\sqrt{x}}\right) + X\left(ye^{-2\sqrt{x}}\right)}{X\left(y e^{2\sqrt{x}}\right) - X\left(ye^{-2\sqrt{x}}\right)}\, p_xp_y , 
                \end{aligned}
        \end{equation} 
where $X(t)$ is an analytic function near zero, and $X'(0) \neq 0$;

        \item[{\rm (C1.3)}]  $m=1$, $s>2$:
        \begin{equation}
        \label{eq:C1.3} 
        \begin{aligned}            
        g &= \frac{
        X\left(y\bigl(1 +  2\sqrt{x}y^{\frac{s-2}{2}}\bigr)^{\frac{2}{2-s}}\right) - 
        X\left(y\bigl(1 -  2\sqrt{x}y^{\frac{s-2}{2}}\bigr)^{\frac{2}{2-s}}\right)}
        { \frac{2}{2-s}\sqrt{x} \, y^{\frac{s}{2}}} \, \dd x \dd y, \\ 
        F &= xp^2_x 
        + \left(\tfrac{2}{2-s}\right)^2y^sp^2_y
        - \tfrac{4\sqrt{x}}{2-s} \, y^{\frac{s}{2}} \frac{
        X\left(y \bigl( 1 {+}  2\sqrt{x}y^{\frac{s-2}{2}}\bigr)^{\frac{2}{2-s}}\right) + 
        X\left( y\bigl(1 {-}  2\sqrt{x}y^{\frac{s-2}{2}}\bigr)^{\frac{2}{2-s}}\right)}
        {X\left(y\bigl(1 {+}  2\sqrt{x}y^{\frac{s-2}{2}}\bigr)^{\frac{2}{2-s}}\right) - 
        X\left(y\bigl(1 {-}  2\sqrt{x}y^{\frac{s-2}{2}}\bigr)^{\frac{2}{2-s}}\right)}p_xp_y , 
        \end{aligned}
        \end{equation} 
 where $X(t)$ is an analytic function near zero, $X'(0) \neq 0$, $s > 2$;

        \item[{\rm (C2)}]  $m=1$, $s=2k$, $k\ge 2$:
        \begin{equation}
        \label{eq:C2} 
        \begin{aligned}    
        g &= \frac{1 + y^{k-1}}{\sqrt{x} \, y^k} 
        \Big( 
        X\left(\rho\left(2\sqrt{x}, y\right)\right) - 
        X\left(\rho\left(-2\sqrt{x}, y\right)\right) \Big) \, \dd x \dd y, \\
        F &= xp^2_x - 2 \sqrt{x} \frac{y^k}{1 + y^{k-1}} 
        \frac{
        X\left(\rho\left(2\sqrt{x}, y\right)\right) + 
        X\left(\rho\left(-2\sqrt{x}, y\right)\right)}
        {X\left(\rho\left(2\sqrt{x}, y\right)\right) - 
        X\left(\rho\left(-2\sqrt{x}, y\right)\right)} \, p_xp_y + \frac{y^{2k}}{(1 + y^{k-1})^2} \, p^2_y, 
        \end{aligned}
        \end{equation} 
 where  $X(t)$ is an analytic function near zero, $X'(0) \neq 0$ and and  $\rho(x,y)$ is the solution of \eqref{eq:forrho} with $f(y)=\frac{y^k}{1 + y^{k-1}}$;

        \item[{\rm (C3)}]  $m=2$, $s=2$:
        \begin{equation}
        \label{eq:C3} 
        g = h(xy)\dd x \dd y, \quad F = (xp_x - yp_y)^2 + \operatorname{const} \frac{p_xp_y}{h(xy)}, 
        \end{equation}
        $h(t)$ is an analytic function near zero, and $h(0) \neq 0$.
        
           \item[{\rm (C4)}]  $D\equiv 0$, $a(x)\equiv 0$,  $c(0)=0$, $c'(0)\ne 0$:         
           \begin{equation}
        \label{eq:C4} 
   g = 2\dd x \dd y, \quad F = yp^2_y - (x+C)p_xp_y, \quad C\in\R.
    \end{equation}
    \end{itemize}
    
    In all the cases, the singular point $\mathsf p$ has coordinates $(x,y)=(0,0)$. 
\end{Theorem}

\begin{proof}  We start with Cases (C1) and (C2) from Proposition \ref{int norm forms}, that is,
$g = 2\lambda (x,y)\dd x \dd y$ and $F = x p_x^2 + b(x,y) p_x p_y + c(y) p_y^2$.  

Let us make a (singular) transformation $x=\frac{\widetilde x^2}{4}$    (in the domain $x>0$) which brings $g$ and $F$ to the following form
$$
g = 2\, \widetilde \lambda (\widetilde x, y) \, \dd \widetilde x \,\dd y
\quad\mbox{and}\quad
F =  p_{\widetilde x}^2 + \widetilde b (\widetilde x, y)  p_{\widetilde x} \, p_y + c(y) p_y^2
$$
where $\widetilde \lambda (\widetilde x, y) = \frac{\widetilde x}{2} \, \lambda \left(\tfrac{\widetilde x^2}{4} , y\right)$ and 
$\widetilde b (\widetilde x, y) = \tfrac{2}{\widetilde x}\, b\left(\tfrac{\widetilde x^2}{4},y\right)$.   The description of such pairs $g$ and $F$ is given in Proposition \ref{prop:2}  and,  more specifically for $g$ which becomes degenerate for $\widetilde x =0$, in Remark \ref{rem:3}.  Hence we conclude that 
$$
\tfrac{\widetilde x}{2} \, \lambda \left(\tfrac{\widetilde x^2}{4} , y\right) = \widetilde \lambda (\widetilde x, y) = 
\frac{X\bigl( \rho(\widetilde x,y)\bigr) - X\bigl( \rho(-\widetilde x,y)\bigr)}{f(y)}
$$
and
$$
\tfrac{2}{\widetilde x}\, b\left(\tfrac{\widetilde x^2}{4},y\right) = \widetilde b (\widetilde x, y) =  -2 \frac{X\bigl(\rho (\widetilde x,y)\bigr) + X\bigl(\rho(-\widetilde x, y)\bigr)}{\lambda( \widetilde x,y)}   
$$
where $X'(0)\ne 0$.   In terms of the original variables $x$ and $y$, we have
$$
\lambda \left(x , y\right) =  \lambda \left(\tfrac{\widetilde x^2}{4} , y\right)   = 2 \, \frac{\widetilde \lambda ( \widetilde x, y)}{\widetilde x } =
\frac{X\bigl( \rho(2\sqrt{x},y)\bigr) - X\bigl( \rho(-2\sqrt{x},y)\bigr)}{f(y)\sqrt{x}}.
$$
Notice that $\widetilde \lambda ( \widetilde x, y)$ is real analytic in $\widetilde x$ and $y$  and moreover is odd w.r.t. $\widetilde x$. Hence, this function is {\it divisible} by $\widetilde x$, and therefore 
$\lambda \left(\tfrac{\widetilde x^2}{4} , y\right) = 2\frac{\widetilde \lambda ( \widetilde x, y)}{\widetilde x}$ is real analytic in variables $\widetilde x$, and $y$.  Moreover,  $\lambda \left(\tfrac{\widetilde x^2}{4} , y\right)$ is even w.r.t. $\widetilde x$, which implies that $\lambda (x,y) = \lambda \left(\tfrac{\widetilde x^2}{4} , y\right)$ is real analytic as a function of $x= \frac{\widetilde x^2}{4}$ and $y$.

Similarly, 
$$
 b\left(x,y\right)  = b\left( \tfrac{\widetilde x^2}{4}, y   \right)=- 2 \frac{X\bigl(\rho (\widetilde x,y)\bigr) + X\bigl(\rho(-\widetilde x, y)\bigr)} {\widetilde \lambda(\widetilde x,y) / \widetilde x} = -  2 \sqrt{x} f(y) \frac{X\bigl(\rho (2\sqrt{x},y)\bigr) + X\bigl(\rho(-2\sqrt{x}, y)\bigr)}{X\bigl(\rho ( 2\sqrt{x},y)\bigr) - X\bigl(\rho(-2\sqrt{x}, y)\bigr)}.
$$

Here the numerator $X\bigl(\rho (2\sqrt{x},y)\bigr) + X\bigl(\rho(-2\sqrt{x}, y)\bigr)$ and the denominator $\widetilde \lambda(\widetilde x,y) / \widetilde x$ are both real analytic in $\widetilde x$ and $y$ and also even w.r.t. $\widetilde x$.  Moreover, the denominator does not vanish at $(\widetilde x, y) = (0,0)$.  Hence, the whole expression is real analytic in variables $\widetilde x$ and even w.r.t. $\widetilde x$.  Therefore, 
$b\left(x,y\right)  = b\left( \tfrac{\widetilde x^2}{4}, y   \right)$ is real analytic as a function of $x= \frac{\widetilde x^2}{4}$ and $y$.  The above formulas work in the region $x>0$,  but due to analyticity,   for $x\le 0$ also.

To summarise this discussion, we conclude that normal forms for $g$ and $F$ in cases (C1) and (C2) are obtained from the normal forms in cases (B1) and (B2)  by  transformation  $x = \frac{\widetilde x^2}{4}$   and replacing $Y$ with $X$,  assuming that  $X'(0)\ne 0$.

This leads us to formulas  \eqref{eq:C1.2}, \eqref{eq:C1.3} and \eqref{eq:C2}  as corollaries of  \eqref{eq:B1.2}, \eqref{eq:B1.3} and \eqref{eq:B2}  respectively.    To get \eqref{eq:C1.1}  (case (C1.1)), we recall from the proof of Theorem \ref{thm2} that the function $\rho$ in (B1.1) takes the form 
$\rho(x,y)=(\frac{x}{2} + \sqrt{y})^2$.   Hence, $X \bigl(\rho ( 2\sqrt{x},y)\bigr) = X \bigl( ( \sqrt{x} + \sqrt{y} )^2 \bigl)$.  Similar to the proof of Theorem \ref{thm2},  we set  $X \bigl( ( \sqrt{x} + \sqrt{y} )^2 \bigl) = \widetilde X \bigl(\sqrt{x} + \sqrt{y}\bigr)$ where $\widetilde X$ is a real analytic {\it even} function.  

\weg{
It remains to show that the expressions \eqref{eq:C?}  are real analytic in $x$ and $y$ in neighbourhood of $(x,y)=(0,0)$ and explain the condition $X'(0)\ne 0$.   
}

\medskip

Now, consider case (C3) where
$$
g = 2\lambda(x,y) \, \dd x \dd y, \quad F=x^2p^2_x+b(x,y)p_xp_y+y^2p^2_y,
$$

In the domain $x>0$, $y>0$,   consider the coordinate transformation $\widetilde x = \ln x$ and $\widetilde y = \ln y$.   Then 
$$
g = 2\widetilde \lambda(\widetilde x, \widetilde y) \, \dd \widetilde x \dd \widetilde y, \quad F=p^2_{\widetilde x}+ \widetilde b(\widetilde x, \widetilde y)p_{\widetilde x}p_{\widetilde y}+p^2_{\widetilde y},
$$
where $\widetilde \lambda (\ln x, \ln y) = xy \lambda (x,y)$.    In the coordinates $\widetilde x, \widetilde y$,  the integral $F$ is of the form (A1)  from Proposition \ref{int norm forms}  that corresponds to the generic case studied in \cite{bolsinov-matveev-pucacco2009}.   In this case,  $\widetilde \lambda (\widetilde x, \widetilde y)$  takes the form
$$
\widetilde \lambda (\widetilde x, \widetilde y) = \widetilde X(\widetilde x + \widetilde y) + \widetilde Y(\widetilde x - \widetilde y) 
$$
for some functions $X$ and $Y$.     Hence, in this domain we obtain the following formula for $\lambda (x,y)$:
$$
\lambda (x,y) =  \frac{\widetilde X(\ln x + \ln y) + \widetilde Y(\ln x - \ln y)}{xy} =
\frac{\widetilde X(\ln (xy) ) + \widetilde Y(\ln \tfrac{x}{y})}{xy} = \frac{X(xy) +  Y(\tfrac{x}{y})}{xy}, 
$$
where $X(t) = \widetilde X(\ln t)$ and $Y(t) = \widetilde Y(\ln t)$.
In the domain $x>0$, $y>0$, the functions $X$ and $Y$ can be arbitrary.    However, in our case they must be such that the left-hand side of the above formula extends up to a real analytic function in a neighbourhood of $(x,y)=(0,0)$.  It can be easily shown that this condition implies that $Y(\frac{x}{y}) = \mathrm{const}$.  Moreover, without loss of generality we may assume that this constant is zero  (otherwise we simply replace $Y$ with $Y - \mathrm{const}$ and $X$ with $X + \mathrm{const}$).    Thus, $\lambda (x,y) = \frac{X(xy)}{xy}= h(xy)$ for some function $h$ which has to be real analytic  and $h(0) = \lambda(0,0)\ne 0$.    The formula for $b(x,y)$ can be easily obtained by solving \eqref{poisson br} with $\lambda = h(xy)$, $a=x^2$ and $c=y^2$.

\medskip

It remains to treat the last case with the pre-normal form (C4) from Proposition \ref{int norm forms}. We have $g =2 \lambda(x,y) \dd x \dd y$ with quadratic  integral  $F = y p^2_y  + b(x,y) p_xp_y$. 
 
To reconstruct $\lambda$ and $b$ we need to resolve the PDE system \eqref{poisson br} which now takes the form
    $$
    \begin{cases}
    (\lambda b)_y = 0 \\
    (\lambda b)_x + \lambda + 2 y\lambda_y = 0
\end{cases}
    $$
  Differentiating the second equation by $y$ gives:
  $$
  3\lambda_y + 2y\lambda_{yy} = 0
  $$
  The general solution to this equation for $y > 0$ is $\lambda(x,y) = \psi(x) y^{-\frac{1}{2}} + \varphi(x)$ and for $y < 0$ is $\lambda(x,y) = \psi(x) (-y)^{-\frac{1}{2}} + \varphi(x)$, or in general $\lambda(x,y) = \psi(x) |y|^{-\frac{1}{2}} + \varphi(x)$.
  Hence, the only analytic solution is $\lambda(x,y) = \varphi(x)$, where $\varphi(x)$ is analytic in a neighbourhood of zero. Hence, $g = 2\varphi(x) \dd x\dd y$, and changing $x$  we can reduce it to $g=2\,\dd x\dd y$ with $\lambda(x,y) \equiv 1$.  Then the above equation for $b$ becomes simply $b_y = 0$, $b_x + 1=0$.   Hence $b=-x - C$  with some constant $C\in \R$.
Finally, we get $g = 2\dd x \dd y$, $F = yp^2_y - (x+C) p_xp_y$, as stated.
\end{proof}

\section{Examples}

Below we give a few explicit examples of normal forms for $g$ and $F$ from Theorems \ref{thm1} and \ref{thm2}.  They are obtained by choosing the simplest possible parameters $X$, $h$ and $H$ in $B$ and $C$ series.

\begin{itemize}
    \item[(B 1.1)] $m = 0$, $s= 1$; setting $X(t)=t$ in \eqref{eq:B1.1} gives
    $$
    g = 2 \, \dd x\, \dd y, \quad F = p^2_x + y\,p^2_y - x\,p_xp_y.
    $$
    \item[(B 1.2)] $m = 0$, $s = 2$; setting $h \equiv 1$, $H \equiv 0$ in \eqref{eq:B1.1} gives 
    $$
    g = 2 \cosh x \, \dd x \, \dd y, \quad F = p^2_x + y^2p^2_y - 2y \tanh x \, p_xp_y.
    $$
 It is easy to see that $g$ is flat and the transformation $u = \sinh x $ brings $g$ and $F$ to the following form 
$$
g = 2\,\dd u \,\dd y, \quad F = (1 + u^2)\,p^2_u + y^2\, p^2_y - 2 yu\, p_up_y.
$$

    \item[(B 1.3)]  $m = 0$, $s\ge 3$; setting $h \equiv 1$, $H \equiv 0$ in \eqref{eq:B1.3} gives 
   for $s =3$:
$$
    g = 2 \frac{1 + 3x^2 y}{(1 - x^2 y)^3} \, \dd x \, \dd y
    \quad\mbox{and}\quad
    F = p_x^2 + 4y^3 p_y^2 - \frac{4x y^2 (3 + x^2 y)}{1 + 3x^2 y}  p_x p_y,
$$
    and for $s = 4$:
$$
g 
   = 2 \frac{1+x^2y^2}{(1-x^2y^2)^2}\,\dd x\,\dd y,  \quad\mbox{and}\quad
F = p_x^2 + y^4 p_y^2 - \frac{4x y^3}{1+x^2y^2}\,p_x p_y.
$$

For  $s>4$, the formulas become more complicated.

    \item[(B 3.1)] 
    $m = 0, s = 2$; setting $h \equiv 1$, $H \equiv 0$ in \eqref{eq:B3.1} gives
    $$
    g = 2 \cos x \, \dd x \dd y, \quad F = -p^2_x + y^2p^2_y - 2y\tan x \, p_xp_y.
    $$
    
The transformation  $u = \sin(x)$ brings $g$ and $F$ to a simpler form 
$$
g = 2 \, \dd u \dd y, \quad F = (up_u - yp_y)^2 - p^2_u.
$$

    \item[(B 3.2)]
    $k=2$;   setting $h \equiv 1$, $H \equiv 0$ in \eqref{eq:B3.2} gives
    $$
    g = \frac{2(1 - x^2 y^2)}{(1 + x^2 y^2)^2} \, \dd x  \dd y
    \quad
    F = -p_x^2 + y^4 p_y^2 + \frac{4x y^3}{1 - x^2 y^2} \, p_x p_y.
    $$
    
    \end{itemize}

\begin{itemize}
\item[(C1.1)]
$m=1, s= 1$; setting $X(t) = t^2$ in  \eqref{eq:C1.1} gives
$$
g = 4\, \dd x\dd y, \quad F = x\, p^2_x + y\, p^2_y - (x+y)\, p_xp_y. 
$$

\item [(C1.2)]
$m = 1$, $s = 2$, setting $X(t) = t$ in \eqref{eq:C1.2} gives
$$
g = 2 \frac{\sinh\left(2\sqrt{x}\right)}{\sqrt{x}} \, \dd x \dd y, \quad F = x\, p^2_x + y^2p^2_y - 2y\sqrt{x}\coth\left(2\sqrt{x}\right)p_xp_y.
$$

After coordinate change $u = \cosh\left(2\sqrt{x}\right) - 1$, we have
$$
g = 2\,\dd u \dd y, \quad F = (up_u - yp_y)^2 + 2p_u(up_u - yp_y).
$$

\item [(C1.3)] 
$m = 1$, $s \ge 3$; setting $X(t) = t$ in \eqref{eq:C1.3} gives for $s=3$
$$
g = \frac{4}{(1-4xy)^2}\,\dd x\dd y,
\quad
F = x \,p_x^2 + 4y^3 p_y^2 - (y + 4xy^2)p_x p_y,
$$
and for $s=4$
$$
g = \frac{4}{1-4xy^2}\,\dd x\dd y, 
\quad
F = x p_x^2 + y^4 p_y^2 - y p_x p_y.
$$

\item[(C3)] $m = 2, s = 2$; setting $h\equiv1$ in \eqref{eq:C3} gives
$$
g = \dd x \dd y, \quad F = (xp_x - yp_y)^2 = x^2p^2_x - 2xyp_xp_y + y^2p^2_y. 
$$

\end{itemize}

It is not very surprising that the simplest examples in the series from Theorems \ref{thm2} and \ref{thm3} lead to a flat metric $g$.  However, flat metrics may also appear for a `non-trivial' choice of functions $X$, $h$ and $H$. Here is one of such examples.  Consider the series (C1.1) given by \eqref{eq:C1.1} with $X = \cosh t - 1$, then 
\begin{equation}
        \label{eq:C1.1add}                           
        g = 2 \frac{\sinh\sqrt{x} \sinh\sqrt{y}}{\sqrt{x}\sqrt{y}}\, \dd x \dd y, \quad
         F = x\, p^2_x  + y\, p^2_y - 2 \sqrt{x}\sqrt{y} \frac{\cosh\sqrt{x} \cosh \sqrt{y} - 1}{\sinh \sqrt{x} \sinh \sqrt{y}}p_xp_y.
\end{equation} 
The transformation $u = \cosh \sqrt{x} -1$,  $v = \cosh \sqrt{y} -1$ brings $g$ and $F$ to the standard form 
$$
g = 2\dd u \dd v, \quad F = (2u + u^2)p^2_u - 2(u + v + uv)p_up_v + (2v + v^2)p^2_v
$$


\section{Summary}\label{sect5}

This section summarises local classification results on geodesically equivalent  Riemannian and pseudo-Riemannian metrics in dimension 2.  Recall that in dimension 2, the description of pairs $(g, \bar g)$ of geodesically equivalent  metrics reduces to a description of metrics $g$ whose geodesic flows admit a quadratic integral $F$  (see  \cite{matveev-topalov2001} and Section \ref{sect1}).  The relation between $F$ and $\bar g$ is given by 
\begin{equation}\label{metric}
    \bar g \;=\; \frac{1}{(\det g\;\det F)^2}\;g\,F\,g.
\end{equation}
Below we list normal forms for $(g, F)$  so that the corresponding normal forms for $(g, \bar g)$ follow automatically from \eqref{metric}.

Let  $F$ be a nontrivial quadratic integral of the geodesic flow of a metric $g$ given in a neighbourhood of $\mathsf p = (0,0)\in\R^2$, and $L =  \det g \cdot \det F \cdot g^{-1}F^{-1}$  be an operator whose algebraic structure allows us to distinguish different types of generic and singular points as explained in Section \ref{sect2}. Then there exist coordinates $(x,y)$ in which $g$ and $F$ (perhaps up to rescaling $F \mapsto \mathrm{const}\cdot F$) take one of the following normal forms depending on the algebraic properties of $L$. 

\medskip 

\textbf{Riemannian case.}

\begin{itemize}
\item $\mathsf p$ is generic, $L$ has two different real eigenvalues: classical result from 1869 U.~Dini \cite{dini1869}, in modern language \cite{bolsinov-matveev-fomenko1998} (Theorem 5).

    $$
g = \Big( f(x) + g(y) \Big)(\dd x^2 + \dd y^2), \quad F = \frac{-f(x)p^2_y + g(y)p^2_x}{f(x) + g(y)},
$$
where $f(x), g(y)$ are smooth functions, $f(x) + g(y) \ne 0$.

\item $\mathsf p$ is singular,  $L(\mathsf p)$ is scalar, while generically $L$ has two different non-constant real eigenvalues: see \cite{bolsinov-matveev-fomenko1998} (Theorem 6, part 2):

     $$
g = \frac{h(x + r) - h(x - r)}{2r} (\dd x^2 + \dd y^2), \quad F = - r\frac{h(x + r) + h(x - r)}{h(x + r) - h(x - r)}(p^2_x + p^2_y) + (xp^2_x + 2yp_xp_y - xp^2_y),
$$
where $r = \sqrt{x^2+y^2}$, $h(t)$ is a smooth function in a neighbourhood of $t=0$ and  $h'(0) > 0$.

\item $\mathsf p$ is singular,  $L(\mathsf p)$ is scalar,  while generically $L$ has one constant and one non-constant eigenvalue: see \cite{bolsinov-matveev-fomenko1998} (Theorem 6, part 1):

$$
g = f(x^2+y^2) (\dd x^2 + \dd y^2), \quad F = (yp_x - xp_y)^2 + \operatorname{const} \frac{p_x^2 + p_y^2}{f(x^2 + y^2)},
$$
where $f(t)$ is a smooth positive function.

\end{itemize}

\textbf{Pseudo-Riemannian case.}

\begin{itemize}

\item  $\mathsf p$ is generic: see \cite{bolsinov-matveev-pucacco2009} Theorem 1.

\begin{itemize}
    \item $L$ has two different real eigenvalues: 
    $$
    g = \big(X(x) - Y(y)\big)(\dd x^2 - \dd y^2), \quad F = \frac{X(x)p^2_y - Y(y)p^2_x}{X(x) - Y(y)},
    $$
    where $X(x)$, $Y(y)$ are smooth functions and $X(x) - Y(y) \ne 0$.

    \item $L$ has two complex conjugate eigenvalues: 
    $$
    g = \operatorname{Im}(h)\, \dd x \dd y, \quad F = p^2_x - p^2_y + 2\, \frac{\operatorname{Re}(h)}{\operatorname{Im}(h)}\, p_x p_y,
    $$
    where $h(x + \ii y)$ is a holomorphic function in variable $z=x + \ii y$, and $\operatorname{Im}h(0) \ne 0$.

    \item $L$ is a Jordan block:
    $$
    g = \big(1 + xY'(y)\big) \dd x \dd y, \quad F = p^2_x - 2\, \frac{Y(y)}{1 + xY'(y)}\, p_x p_y,
    $$
    where $Y(y)$ is a smooth function in some neighbourhood of $\mathsf p$.
\end{itemize}

\item $\mathsf p$ is singular  \footnote{These results are obtained under the assumption that $g$ and $F$ are real analytic.}

\begin{itemize}
    \item $ L(\mathsf p)$ is a Jordan block,  while generically $L$ has two different eigenvalues (real or complex): 

the normal forms are given by formulas \eqref{eq:B1.1}-\eqref{eq:B4} from Theorem \ref{thm2}.

\item $ L(\mathsf p)$ is a scalar operator, while generically $L$ has two different eigenvalues (real or complex): 

the normal forms are given by formulas \eqref{eq:C1.1}-\eqref{eq:C3} from Theorem \ref{thm3}.

\item $ L(\mathsf p)$ is a scalar operator, while generically $L$ is a Jordan block: 

the normal forms are given by \eqref{eq:C4} from Theorem \ref{thm3}.

\end{itemize}

\end{itemize}


\begin{thebibliography}{99}

\bibitem{aminova2003}
A. V. Aminova, \textit{Projective transformations of pseudo-Riemannian manifolds. Geometry, 9}, J. Math. Sci. (N. Y.) 113(2003), no. 3, 367–470.

\bibitem{beltrami1865}
E. Beltrami, \textit{Risoluzione del problema: riportare i punti di una superficie sopra un piano in modo che le linee geodetiche vengano rappresentate da linee rette}, Ann. Mat., 1(1865), no. 7, 185–204.

\bibitem{beltrami1868}
E. Beltrami, \textit{Saggio di interpetrazione della geometria non-euclidea}, Giornale di matematiche, vol. VI(1868).

\bibitem{beltrami1869}
E. Beltrami, \textit{Teoria fondamentale degli spazii di curvatura costante}, Annali di Mat., ser. II, 2(1869), 232–255.

\bibitem{Bolsinov-2025}
A. Bolsinov, \textit{Geodesically Equivalent Metrics and Nijenhuis Geometry}, Sbornik: Mathematics, 2025, 216, 5, 5--32.

\bibitem{BKM-1}
A.V. Bolsinov, A.Y. Konyaev, V.S. Matveev, \textit{Nijenhuis geometry}, Advances in Mathematics, \textbf{394} (2022), 108001.

\bibitem{BKM-3}
A.V. Bolsinov, A.Y. Konyaev, V.S. Matveev, \textit{Nijenhuis geometry III: gl-regular Nijenhuis operators}, Rev. Mat. Iberoam., \textbf{40} (2024), no. 1, 155--188.

\bibitem{nijapp5} 
A. V. Bolsinov, A. Yu. Konyaev, V. S. Matveev, \textit{Applications of Nijenhuis Geometry V: geodesic equivalence and finite-dimensional reductions of integrable quasilinear systems}, Journal of Nonlinear Science (2024) 34:33 DOI: 10.1007/s00332-023-10008-0.

\bibitem{tams} 
A. V. Bolsinov, V. S. Matveev, \textit{Local normal forms for geodesically equivalent pseudo-Riemannian metrics}, Trans. Amer. Math. Soc. {\bf 367} (2015) no. 9, 6719--6749.

\bibitem{eqm} A.\,V.\,Bolsinov, V.\,S.\,Matveev, Geometrical interpretation of Benenti systems, J. Geom. Phys. {\bf 44}(2003), no.~4, 489--506.

\bibitem{bolsinov-matveev-fomenko1998}
A. V. Bolsinov, V. S. Matveev, A. T. Fomenko, \textit{Two-dimensional Riemannian metrics with integrable geodesic flows. Local and global geometry}, Sb. Math., \textbf{189}:10 (1998), 1441–1466.

\bibitem{bolsinov-matveev-pucacco2009}
A. V. Bolsinov, V. S. Matveev, G. Pucacco, \textit{Normal forms for pseudo-Riemannian 2-dimensional metrics whose geodesic flows admit integrals quadratic in momenta}, Journal of Geometry and Physics, \textbf{59}(7), (2009), 1048--1062.


\bibitem{bryant2008}
R. L. Bryant, G. Manno, V. S. Matveev, \textit{A solution of a problem of Sophus Lie: Normal forms of 2-dim metrics admitting two projective vector fields}, Math. Ann. 340(2008), no. 2, 437–463.

\bibitem{bryant2009}
R. L. Bryant, M. Dunajski, M. Eastwood, \textit{Metrisability of two-dimensional projective structures}, J. Diff. Geom. 83(2009), no. 3, 465–499.

\bibitem{dini1869}
U. Dini, \textit{Sopra un problema che si presenta nella teoria generale delle rappresentazioni geografiche di una superficie su un’altra}, Annali di Matematica, Serie 2, \textbf{3} (1869), 269–293.

\bibitem{eastwood2007}
M. Eastwood, V. S. Matveev, \textit{Metric connections in projective differential geometry}, Symmetries and Overdetermined Systems of PDEs, IMA Vol. Math. Appl., 144(2007), 339–351.

\bibitem{eisenhart1923}
L. P. Eisenhart, \textit{The geometry of paths and general relativity}, Ann. of Math. (2) 24(1923), no. 4, 367–392.

\bibitem{eisenhart1927}
L. P. Eisenhart, \textit{Non-Riemannian Geometry}, American Mathematical Society Colloquium Publications VIII(1927).

\bibitem{fubini1903}
G. Fubini, \textit{Sui gruppi transformazioni geodetiche}, Mem. Acc. Torino 53(1903), 261–313.

\bibitem{ilyashenko2013}
Yu. Ilyashenko, S. Yakovenko, \textit{Lectures on Analytic Differential Equations}, AMS, 2007, Graduate Studies in Mathematics, 86.

\bibitem{kiosak2009}
V. Kiosak, V. S. Matveev, \textit{Complete Einstein metrics are geodesically rigid}, Comm. Math. Phys. 289(2009), no. 1, 383–400.

\bibitem{kolokoltsov1983}
V. N. Kolokoltsov, \textit{Geodesic flows on two-dimensional manifolds with an additional first integral that is polynomial in the velocities}, Math. USSR-Izv., \textbf{21}:2 (1983), 291–306.

\bibitem{lagrange1779}
J.-L. Lagrange, \textit{Sur la construction des cartes géographiques}, Noveaux Mémoires de l’Académie des Sciences et Bell-Lettres de Berlin, 1779.

\bibitem{levicivita1896}
T. Levi-Civita, \textit{Sulle trasformazioni delle equazioni dinamiche}, Ann. di Mat., serie 2a, 24(1896), 255–300.

\bibitem{lie1882}
S. Lie, \textit{Untersuchungen über geodätische Kurven}, Math. Ann. 20 (1882); Sophus Lie Gesammelte Abhandlungen, Band 2, erster Teil, 267–374. Teubner, Leipzig, 1935.

\bibitem{liouville1889}
R. Liouville, \textit{Sur les invariants de certaines équations différentielles et sur leurs applications}, Journal de l’École Polytechnique 59 (1889), 7–76.

\bibitem{matveev2007}
V. S. Matveev, \textit{Proof of projective Lichnerowicz-Obata conjecture}, J. Diff. Geom. 75(2007), 459–502.

\bibitem{matveev2012}
V. S. Matveev, \textit{Two-dimensional metrics admitting precisely one projective vector field}, Math. Ann. 352(2012), 865–909.

\bibitem{matveev2012_gr}
V. S. Matveev, \textit{Geodesically equivalent metrics in general relativity}, J. Geom. Phys., 62(2012), no. 3, 675–691.

\bibitem{matveevtopalov1998}
V. S. Matveev, P. J. Topalov, \textit{Trajectory equivalence and corresponding integrals}, Regular and Chaotic Dynamics, 3(1998), no. 2, 30–45.

\bibitem{matveev-topalov2001}
V. Matveev, P. Topalov, \textit{Quantum integrability of Beltrami-Laplace operator as geodesic equivalence}, Math. Z., \textbf{238} (2001), 833--866.

\bibitem{mikes1996} J. Mikeš, \textit{Geodesic mappings of affine-connected and Riemannian spaces. Geometry, 2}, J. Math. Sci. 78(1996), no. 3, 311–333.

\bibitem{painleve1897}
P. Painlevé, \textit{Sur les intégrales quadratiques des équations de la dynamique}, Compt. Rend., 124(1897), 221–224.



\bibitem{thomas1945}
T. Thomas, \textit{On the projective theory of two dimensional Riemann spaces}, Proc. Nat. Acad. Sci. U. S. A. 31(1945), 259–261.

\bibitem{veblen1923}
O. Veblen, T. Thomas, \textit{The geometry of paths}, Trans. Amer. Math. Soc. 2(1923), no. 4, 551–608.

\bibitem{veblen1926}
O. Veblen, J. Thomas, \textit{Projective invariants of affine geometry of paths}, Ann. of Math. (2) 27(1926), no. 3, 279–296.

\bibitem{weyl1921}
H. Weyl, \textit{Zur Infinitisimalgeometrie: Einordnung der projektiven und der konformen Auffassung}, Nachr. d. K. Ges. d. Wiss. zu Göttingen, Math.-Phys. Kl., 1921.

\bibitem{yano1940}
K. Yano, \textit{Concircular Geometry}, I–IV, Proc. Imp. Acad. Tokyo, \textbf{16} (1940).

\bibitem{yano1956}
K. Yano, ``Sur la correspondence projective entre deux espaces pseudohermitiens,'' \textit{C.R. Acad. Sci.}, \textbf{239} (1956), 1346--1348.

\bibitem{yano-nagano1957}
K. Yano and T. Nagano, ``Some theorems on projective and conformal transformations,'' \textit{Koninkl. Nederl. Akad. Wet.}, \textbf{A60}, No. 4 (1957), 451--458.





\end{thebibliography}
\end{document}